\RequirePackage{amsthm}
\documentclass[sn-basic]{sn-jnl}

\usepackage{graphicx}
\usepackage{multirow}
\usepackage{amsmath,amssymb,amsfonts}
\usepackage{amsthm}
\usepackage{mathrsfs}
\usepackage[title]{appendix}
\usepackage{xcolor}
\usepackage{textcomp}
\usepackage{manyfoot}
\usepackage{booktabs}
\usepackage{algorithm}
\usepackage{algorithmicx}
\usepackage{algpseudocode}
\usepackage{listings}
\usepackage{dsfont}
\usepackage{xspace}
\usepackage{bbm}
\usepackage{soul}

\newcommand{\norm}[1]{\left\lVert#1\right\rVert}
\newcommand{\E}{\mathbb{E}}
\newcommand{\Bn}{\ensuremath{\mathcal{B}_n}\xspace}
\newcommand{\Sn}{\ensuremath{\mathcal{S}_n}\xspace}
\newcommand{\Cb}{\ensuremath{\mathcal{C}_\mathcal{B}}\xspace}
\newcommand{\Cs}{\ensuremath{\mathcal{C}_\mathcal{S}}\xspace}
\newcommand{\Yit}{\ensuremath{\tilde{Y_i}}\xspace}
\newcommand{\Zit}{\ensuremath{\tilde{Z_i}}\xspace}

\newcommand{\sigzm}{\ensuremath{{\sigma_Z^2}}_-\xspace}
\newcommand{\sigza}{\ensuremath{{\sigma_Z^2}}_+\xspace}

\newcommand{\sigytm}{\ensuremath{{\sigma_{\tilde{Y}}^2}_-}\xspace}
\newcommand{\sigyta}{\ensuremath{{\sigma_{\tilde{Y}}^2}_+}\xspace}

\newcommand{\muytppm}{\ensuremath{{\mu''_{\tilde{Y}}}_-}\xspace}
\newcommand{\muytppa}{\ensuremath{{\mu''_{\tilde{Y}}}_+}\xspace}
\newcommand{\muyppm}{\ensuremath{{\mu''_{Y}}_-}\xspace}
\newcommand{\muyppa}{\ensuremath{{\mu''_{Y}}_+}\xspace}

\newcommand{\tauh}{\hat{\tau}_h}
\newcommand{\tauy}{\tau_Y}

\newcommand{\muzpm}{\ensuremath{{\mu'_{Z}}_-}\xspace}
\newcommand{\muzpa}{\ensuremath{{\mu'_{Z}}_+}\xspace}
\newcommand{\ind}{\mathds{1}(x\geq 0)}
\DeclareMathOperator*{\argmin}{argmin}
\newcommand{\thetanhat}{\hat{\theta}_n}
\newcommand{\gammanhat}{\hat{\gamma}_n}
\newcommand{\thetancheck}{\check{\theta}_n}
\newcommand{\gammancheck}{\check{\gamma}_n}
\newcommand{\thetazerocheck}{\check{\theta}_0}
\newcommand{\gammazerocheck}{\check{\gamma}_0}
\newcommand{\richeck}{\check{r}_i}
\newcommand{\Kpn}{K_+^{(0)}}
\newcommand{\Kpe}{K_+^{(1)}}
\newcommand{\Kpz}{K_+^{(2)}}
\newcommand{\Kpd}{K_+^{(3)}}
\newcommand{\Kmn}{K_-^{(0)}}
\newcommand{\Kme}{K_-^{(1)}}
\newcommand{\Kmz}{K_-^{(2)}}
\newcommand{\Kmd}{K_-^{(3)}}
\newcommand{\Kn}{K^{(0)}}
\newcommand{\Ke}{K^{(1)}}
\newcommand{\Kz}{K^{(2)}}
\newcommand{\Kd}{K^{(3)}}

\newcommand{\Khfrac}{\mathbf{K}_h^{\frac{1}{2}}}
\newcommand{\Kh}{\mathbf{K}_h}
\newcommand{\Zbf}{\mathbf{Z}}
\newcommand{\Vbf}{\mathbf{V}}
\newcommand{\Ybf}{\mathbf{Y}}
\newcommand{\rbf}{\mathbf{r}}
\newcommand{\Ztbf}{\mathbf{Z}^\top}

\newcommand{\IR}{\mathbb{R}}
\newcommand{\IE}{\mathbb{E}}
\newcommand{\IP}{\mathbb{P}}
\newcommand{\Ind}{\mathbbm{1}}

\renewcommand{\bar}{\overline}
\renewcommand{\epsilon}{\varepsilon}
\renewcommand{\phi}{\varphi}

\theoremstyle{thmstyleone}
\newtheorem{theorem}{Theorem}
\newtheorem{proposition}{Proposition}
\newtheorem{lemma}{Lemma}[section]

\theoremstyle{thmstylethree}
\newtheorem{definition}{Definition}
\newtheorem{assumption}{Assumption}

\begin{document}

\title{Estimating Average Treatment Effects in Regression Discontinuity Designs with Covariates under Minimal Assumptions}

\author*[1]{\fnm{Patrick} \sur{Kramer}}\email{pk25wyda@studserv.uni-leipzig.de}

\author[1]{\fnm{Alexander} \sur{Kreiß}}\email{alexander.kreiss@math.uni-leipzig.de}

\affil[1]{\orgdiv{Institute of Mathematics}, \orgname{Leipzig University}, \orgaddress{\street{Augustusplatz 10}, \city{Leipzig}, \postcode{04109}, \country{Germany}}}

\abstract{We study regression discontinuity designs with the use of additional covariates for estimation of the average treatment effect. We provide a detailed proof of asymptotic normality of the covariate-adjusted estimator under minimal assumptions, which may serve as an accessible text to the mathematics behind regression discontinuity with covariates. In addition, this proof carries at least three benefits. First of all, it allows to draw straightforward consequences concerning the impact of the covariates on the bias and variance of the estimator. In fact, we can provide conditions under which the influence of the covariates on the bias vanishes. Moreover, we show that the variance in the covariate-adjusted case is never worse than in the case of the baseline estimator under a very general invertibility condition. Finally, our approach does not require the existence of potential outcomes, allowing for a sensitivity analysis in case confounding cannot be ruled out, e.g., by a manipulated forcing variable.}

\keywords{Regression discontinuity design, covariates, average treatment effect estimation, local polynomial regression}

\maketitle

\section{Introduction}\label{introduction}

Regression discontinuity design (RDD) is regularly used in economics, political science and other social sciences like biomedical science \citep{RDDWithCov}. The aim of RDD is to estimate causal effects of a treatment on some outcome in settings where the treatment is applied depending on some running variable (also called forcing variable or score value). Therefore, given data needs to contain a score value for each unit which decides whether the treatment is applied to that unit or not. In fact, if the score value lies above a certain cutoff point, the treatment is applied, otherwise it is not. Then, the effect of the treatment on the outcome can be estimated by comparing units that are close to the cutoff, but are on different sides of it. This concept, including its theoretical properties and applications, has already been studied by a variety of researchers.
\cite{RDFoundations} and \cite{RDExtensions} provide a two-part guide for the analysis and interpretation of regression discontinuity designs. Also the work of \cite{LeeDavid2010} can be seen as a guide for empirical researchers on RDDs, discussing basic theory, different methods of estimation and presenting examples to the reader. \cite{IMBENS2008615} summarize practical and theoretical issues that one encounters when implementing RD methods. \cite{ComparingInference} offer an empirical comparison of different inference approaches, such as parametric and non-parametric local polynomial methods, as well as local randomization methods. \cite{Hahn2001} present non-parametric estimation of the average treatment effect under minimal assumptions, as well as an interpretation of the Wald estimator as an RD estimator. Focusing on local linear regression, \cite{RDMSEBandwidth} provide an asymptotically optimal bandwidth choice under squared error loss. In \cite{Armstrong_2018}, confidence intervals for linear functionals of regression functions, e.g. the regression discontinuity parameter, are constructed, carrying optimal finite-sample properties and sharp efficiency bounds under normal errors with known variance. \cite{KinkRDD} address the issue of biased data-driven confidence intervals, presenting a more robust confidence interval estimator for average treatment effects in different RDD settings.

In practice, there is often more data than only the score value and the outcome. For this reason, many researchers want to incorporate additional data into the analysis of the causal effect. Indeed, this can be done by including additional covariates linearly in a local linear regression around the cutoff as suggested by \cite{RDDWithCov}. In this paper, they provide a formal study of the impact of covariate adjustments on estimation and inference in RD designs. In order to account for misspecification in finite samples \citep{KinkRDD}, they established a bias-corrected version of the covariate-adjusted estimator. In addition to that, smoothness assumptions that are necessary to obtain a consistent estimator that delivers reliable results are given. Under this imposed regularity conditions, they showed that the bias-corrected covariate-adjusted estimator is asymptotically normally distributed.

In this paper we continue the study of asymptotic normality of the regression discontinuity estimator using covariates. More precisely, we simplify the proof strategy of \cite{RDWithPotManyCov} and use that in typical regression discontinuity designs using locally linear estimation, one estimates
$$\tau_0:=\argmin_{\tau}\min_{a,b,c\in\IR,\gamma\in\IR^p}\IE\left(\left(Y_i-a-\tau T_i-bX_i-cT_iX_i-\gamma^TZ_i\right)^2\right),$$
where $X_i$ denotes the running variable, $T_i:=\Ind(X_i\geq c)$ denotes the treatment indicator (where $T_i=1$ means that unit $i$ was treated, hence we study sharp RD with treatment threshold $c$), and $Z_i\in\IR^p$ is an observed, additional vector of covariates (in contrast to \cite{RDWithPotManyCov} we let $p$ be fixed). Taking this view-point allows us to reprove the asymptotic normality result of \cite{RDDWithCov} in a situation without potential outcomes, i.e., we do not require $Y_i=T_iY_i(1)+(1-T_i)Y_i(0)$ for some potential outcomes $Y_i(0),Y_i(1)$, and under no assumptions on the bandwidth $h$ other than the standard assumptions of $h\to0$ and $nh\to\infty$. The main contribution of this paper is therefore to provide a proof of asymptotic normality in regression discontinuity designs with covariates under minimal assumptions. Such a result is useful for at least three reasons:  First, the variance is obtained in a form that allows us to conclude that it is less or equal to the variance of the baseline estimator under a very general invertibility condition, which is a slightly stronger result than the one stated in \cite{RDDWithCov}. Second, the representation of the bias allows to state straightforward adjustments and assumptions under which the leading term of the bias is independent of the covariates, dispelling the concern that the bias might increase when incorporating covariates. Finally, we are able to conduct a sensitivity analysis in case confounding cannot be ruled out.

Furthermore, this paper can serve as an accessible introduction to studying regression discontinuity under the influence of covariates. Since we tackle the situation of a fixed number of covariates, the idea of proving asymptotic normality is not diluted by more complicated arguments. In fact, our provided proof is fairly comparable to \cite{Hahn2001} in terms of its structure as well as the required assumptions, and can be seen as a canonical extension to the situation with incorporated covariates.

This paper contributes to the wide area of literature considering regression discontinuity designs using covariates. In addition to \cite{RDDWithCov} and \cite{RDWithPotManyCov}, this includes the following research: \cite{Froe2019} propose a fully non-parametric kernel method to account for observed covariates in RDDs. \cite{Noack2021} develop a novel class of estimators, opting for a more efficient estimation than linear adjustment estimators. Finally, \cite{Arai2021} study the case of potentially high-dimensional covariates, proposing estimation and inference methods with covariate selection that use a local approach with kernel weights in combination with $\ell_1$-penalization.

The remainder of this paper is structured as follows. Section \ref{setup} introduces basic terminology and the setup. Section \ref{asymptoticnormality} states the main theorem on asymptotic normality of the covariate-adjusted estimator and provides a proof. Moreover, consequences concerning the representation of the bias and variance, as well as a sensitivity analysis, are discussed. Finally, Section \ref{proofsection} entails all statements that were used in the proof of the main theorem. This particularly includes four propositions, each elaborated in its own sub-chapter. All additionally necessary supporting statements are provided in the supplementary appendix of this paper.

\section{Setup and Notation}\label{setup}

The data is given by an independent sample $\{(Y_i,X_i,Z_i), i=1,...,n\}$ whereby $n$ is the population size, $Y_i\in\mathbb{R}$ is the outcome variable, $X_i\in\mathbb{R}$ the running variable with its probability density function $f_X$ and $Z_i\in\mathbb{R}^p$ is a vector of pre-treatment covariates. The running variable $X_i$ decides whether a unit receives a treatment or not. In fact, we consider so called sharp RDDs and define the treatment variable $T_i$ by $T_i:=\mathds{1}(X_i\geq c)$, which means that a treatment is applied if and only if the running variable exceeds a certain cutoff $c$. Without loss of generality and for the sake of simplicity we want this cutoff to be zero throughout this paper, that is $c=0$ and $T_i=\mathds{1}(X_i\geq 0)$.
Moreover, with potential outcomes we mean the random variables $Y_i(0)$ and $Y_i(1)$, which are the outcome in case a unit receives a treatment or not, respectively. In case of their existence (which we do not require), we have that $Y_i=Y_i(T_i)$, which is the real outcome depending on the value of the treatment variable.

The target of our estimation is $\tau_Y=\lim_{x\searrow 0} \IE(Y_i\mid X_i=x)-\lim_{x\nearrow 0} \IE(Y_i\mid X_i=x)$, in other words, the height of the jump of $\mu_Y(x):=\IE(Y_i\mid X_i=x)$ at $x=0$. On the one hand, we note that under the classical assumption $Y_i=T_iY_i(1)+(1-T_i)Y_i(0)$ with $\IE(Y_i(t)\mid X_i=x)$ being continuous at $x=0$ for $t\in\{0,1\}$, $\tau_Y=\IE(Y_i(1)\mid X_i=0)-\IE(Y_i(0)\mid X_i=0)$ can be interpreted as an average treatment effect at the cut-off. On the other hand, however, manipulation of the running variable can lead to discontinuities in $\IE(Y_i(1)\mid X_i=x)$ as pointed out by \cite{GRR20}. We do neither insist on continuous conditional expectations, nor consider a specific confounding mechanism, and instead allow for general outcomes $Y_i$.

For arbitrary random variables $A$ and $B$, we want to define $\mu_A(x):=\E(A\mid X_i=x)$, $\sigma^2_{AB}(x):=\mu_{AB^\top}(x)-\mu_A(x)\mu_B(x)^\top$ and $\sigma^2_A(x):=\sigma^2_{AA}(x)$. Moreover, given a function $f$ defined on the real line for which the left- and right-sided limit in zero exists, we set $f_+=\lim_{x\searrow 0} f(x)$ and $f_-=\lim_{x\nearrow 0} f(x)$. Then we can also represent the average treatment effect, the object of interest, by $\tau_Y=\mu_{Y+}-\mu_{Y-}$. To estimate $\tau$, we will use a local linear adjustment estimator: Let $K$ denote a kernel function, $K_h(x):=\frac{1}{h}K\left(\frac{x}{h}\right)$, $h>0$ the bandwidth and $V_i:=(1, T_i, X_i/h, T_iX_i/h)^\top$. For the sake of notational simplicity we will just write $V_i$ instead of $V_i(h)$ and omit its dependency of $h$. However, it should be noted that this dependency exists and also needs to be considered when looking at the convergence of expressions containing $V_i$. The estimator is then given by
\begin{equation*}
    \tauh=e_2^\top\big(\thetanhat,\gammanhat\big),
\end{equation*}
where $e_2:=\begin{pmatrix}0 & 1 & 0 \dots & 0\end{pmatrix}^\top\in\IR^{4+p}$ and
\begin{equation*}
    \big(\thetanhat,\gammanhat\big)=\argmin_{(\theta, \gamma)\in\mathbb{R}^{p+4}} \sum_{i=1}^n K_h(X_i)\left(Y_i-V_i^\top\theta-Z_i^\top\gamma\right)^2.
\end{equation*}
If we use the term \textit{kernel} in the remainder of this paper, we refer to a function which is defined as follows.
\begin{definition}[Kernel function]\label{defofkernel}
    A kernel is a function $K:\mathbb{R}\to\mathbb{R}$ which
    \begin{itemize}
        \item is non-negative and integrable,
        \item integrates to one, i.e., $\int_{-\infty}^\infty K(x)\,dx=1$, and
        \item is symmetric, i.e., $K(-x)=K(x)$.
    \end{itemize}
\end{definition}

In order to proof the asymptotic normality later on, we will have to deal with the regression error $r_i(h)=Y_i-V_i^\top\theta_0(h)-Z_i^\top\gamma_0(h)$, whereby $\theta_0(h)$ and $\gamma_0(h)$ are the population regression coefficients defined by
\begin{equation*}
    \big(\theta_0(h),\gamma_0(h)\big)=\argmin_{(\theta,\gamma)\in\mathbb{R}^{p+4}} \E\left(K_h(X_i)\left(Y_i-V_i^\top\theta-Z_i^\top\gamma\right)^2\right).
\end{equation*}

We can also define the average treatment effect on random variables other than the outcome as follows:
\begin{definition}[Average treatment effect]
Let $m\in\mathbb{N}$ and $A\in\mathbb{R}^m$ be a random variable such that the left- and right-sided limit of $\mu_A$ in zero exist. Then we define the average treatment effect of $T_i$ on $A$ as
\begin{equation*}
    \tau_A=\mu_{A+}-\mu_{A-}.
\end{equation*}
\end{definition}

In the following, we give some more definitions that are relevant for the subsequent content of this paper.

\begin{definition}
Let $L:\mathbb{R}\to \mathbb{R}$ be an arbitrary integrable function. Then we define
\begin{equation*}
    L_-^{(\alpha)}=\int_{-\infty}^0 L(u)u^\alpha\,du,\quad L_+^{(\alpha)}=\int_0^\infty L(u)u^\alpha\,du\quad \text{and}\quad L^{(\alpha)}=\int_{-\infty}^\infty L(u)u^\alpha\,du.
\end{equation*}
\end{definition}

\begin{definition}
Let $K$ be a kernel function with $K^{(2)}<\infty$. Then we define the matrix
\begin{equation*}
    \kappa(K):=\begin{pmatrix}
    \Kn&\Kpn&\Ke&\Kpe\\
    \Kpn&\Kpn&\Kpe&\Kpe\\
    \Ke&\Kpe&\Kz&\Kpz\\
    \Kpe&\Kpe&\Kpz&\Kpz
    \end{pmatrix}.
\end{equation*}
\end{definition}

The following definition will be used to state the regularity assumptions in the assertions of this paper more compactly and concisely.
\begin{definition}\label{defofonesideddiff}
    We call a function $f:\mathbb{R}\setminus\{0\}\to\mathbb{R}$ $k$-times one-sided differentiable at 0 when $f$ can be extended continuously to zero from the left and from the right, and all derivatives of $f$ up to order $k$ exist on $(-\infty,0)$ and $(0,\infty)$ and can be extended continuously to zero from the left and from the right. Then we denote the value of the continuous extension of $f$ to zero from the left by $f_-$ and from the right by $f_+$. In an analogous manner, we use that notation for the extension of the derivatives, such as $f'_-$ or $f'_+$.
\end{definition}

\section{Asymptotic Normality of the Estimator}\label{asymptoticnormality}

This section aims to prove the main theorem of this paper, stating that the covariate-adjusted estimator of the average treatment effect is asymptotically normally distributed. As several conditions are required such that this holds, the next subsection will briefly cover all necessary assumptions. After that, the final theorem will be stated and proved, followed by its implications regarding the comparison of the bias and variance to the case of estimation without covariates, as well as a sensitivity analysis in case of a confounded forcing variable.

\subsection{Assumptions}\label{secAssump}
In order to state the asymptotic normality of the covariate-adjusted estimator of the average treatment effect, we need to formulate several assumptions that all need to be satisfied.

\begin{assumption}[Convergence of bandwidth]\label{assumption1}
Let the bandwidth $h=h_n$ be a sequence $(h_n)_{n\in\mathbb{N}}$ with $h_n>0$ for all $n\in\mathbb{N}$ such that $h\to 0$ and $nh\to \infty$ for $n\to\infty$.
\end{assumption}

Assumption \ref{assumption1} states that the bandwidth $h$ converges to zero with increasing number of observations, however, slow enough such that $nh$ still goes to infinity. This is a standard assumption when dealing with kernel regressions.

\begin{assumption}[Kernel function]\label{assumption2}
Let $K:\mathbb{R}\to\mathbb{R}$ be a kernel function (cf. Definition \ref{defofkernel}) with $K^{(4)}, \left(K^2\right)^{(2)}<\infty$ which is supported on $[-1,1]$.
\end{assumption}

These conditions are standard assumptions on the kernel. The finiteness of $K^{(4)}, \left(K^2\right)^{(2)}<\infty$ indeed can also be replaced by additionally requiring continuity. In either case, typical kernels, such as triangular and Epanechnikov kernels, are covered under this assumption. Kernels with unbounded support, like the Gaussian kernel, could be used by enhancing the theorem's proof with some slightly more complex arguments.

\begin{assumption}[Differentiability]\label{assumption3}
Let $\mu_Z$ be continuous. Further, let the conditional expectations $\mu_Z$ and $\mu_Y$ be three times one-sided differentiable at zero (cf. Definition \ref{defofonesideddiff}). Also, let $\mu_{ZZ^\top}$ and $\mu_{ZY}$ be one time one-sided differentiable in zero.
\end{assumption}

This differentiability assumption is used to perform Taylor expansions up to order two and obtain remainder terms of order $O(h^3)$.

In order to state the next assumptions, the following two definitions regarding an adjusted version of the covariates and a covariate-adjusted outcome are necessary.

\begin{definition}
    Define the matrix $M_n\in\mathbb{R}^{4\times p}$ as
    \begin{equation*}
        M_n := \begin{pmatrix}
            \mu_Z(0) & \mathbf{0} & h\muzpm & h\left(\muzpa-\muzpm\right)
        \end{pmatrix}^\top.
    \end{equation*}
\end{definition}

\begin{definition}[Adjusted variables]\label{covadjustproc}
    Define
    \begin{equation*}
        \Zit:=Z_i-M_n^\top V_i\quad\text{ and }\quad\Yit:=Y_i-Z_i^\top\tilde\gamma
    \end{equation*}
    with $\tilde\gamma:=\left(\sigzm+\sigza\right)^{-1}\left(\sigma^2_{ZY-}+\sigma^2_{ZY+}\right)$
\end{definition}

The existence of the inverse in the definition of $\tilde\gamma$ is assumed in the next assumption.

\begin{assumption}[Invertibility]\label{assumption4}
Let $\E\left(K_h(X_i)\Zit\tilde Z_i^\top\right)$ and $\sigzm+\sigza$ be invertible and
\begin{equation*}
    \norm{\E\left(K_h(X_i)\Zit\tilde Z_i^\top\right)^{-1}}_2=O(1).
\end{equation*}
\end{assumption}

In fact, we assume invertibility of a local version of the design matrix which is typical for regression.

\begin{assumption}[Density]\label{assumption5}
Let the probability density function $f_X$ of $X_i$ be three times continuously differentiable in a neighborhood around zero with $f_X(0)>0$.
\end{assumption}

This assumption is particularly relevant in order to have coinciding left and right limits of all derivatives up to order three at zero. Also it allows to find upper bounds for all those derivatives on a compact interval. Indeed we will often be in the case of having a compact interval since integration mostly involves the kernel function which has a compact support by assumption.

\begin{assumption}[Uniform boundedness]\label{assumption6}
Let the following statements be true for all \mbox{$k,l\in\{1,...,p\}$}
\begin{align*}
    &\sup_{n\in\mathbb{N}}\sup_{x\in[-h,h]}\, \E\left(\left(\tilde Z_i^{(k)}r_i(h)\right)^2\;\,\vrule\; X_i=x\right) < \infty,\\
    &\sup_{n\in\mathbb{N}}\sup_{x\in[-h,h]}\, \E\left(\left(\tilde Z_i^{(k)}\tilde Z_i^{(l)}\right)^2\;\,\vrule\; X_i=x\right) < \infty
\end{align*}
and let $\delta>0$ such that
\begin{equation*}
    \sup_{n\in\mathbb{N}}\sup_{x\in[-h,h]}\, \E\left(|r_i(h)|^{2+\delta}\mid X_i=x\right) < \infty.
\end{equation*}
\end{assumption}

These uniform boundedness assumptions are needed to find upper bounds of the means irrespective of the value of the running variable or the number of observations. In particular, it allows us to conclude integrability over $x\in[-h,h]$ later on.
Please note that if we consider the simpler case when the outcome variable has the form $Y_i=V_i^\top\bar\theta+Z_i^\top\bar\gamma+\epsilon_i$ for some fixed $h$ with, e.g., zero-mean, normal $\epsilon_i$ independent of $X_i$, $Z_i$, then the third statement of this assumption is satisfied due to $r_i(h)=\epsilon_i$. Moreover, if we additionally assume boundedness of the covariates, also the first two statements of the assumption are immediately true.

\begin{assumption}[Equi-continuity]\label{assumption7}
Let $\sigma_l, \sigma_r>0$ be finite numbers such that
\begin{align*}
    &\lim_{n\to\infty}\sup_{\lambda\in [0,1]} \left|\E\left(r_i(h)^2\mid X_i=\lambda h\right)-\sigma_r^2\right|=0,\\
    &\lim_{n\to\infty}\sup_{\lambda\in [0,1]} \left|\E\left(r_i(h)^2\mid X_i=-\lambda h\right)-\sigma_l^2\right|=0.
\end{align*}
\end{assumption}

The existence of those limits will be used to verify the convergence condition of Lyapunov's Central Limit Theorem, when proving the convergence to the standard normal distribution. Again, this assumption is satisfied if we examine an outcome of the form $Y_i=V_i^\top\bar\theta+Z_i^\top\bar\gamma+\epsilon_i$ as described previously.

\subsection{Main Theorem}

The representation of the bias and the variance that are used in the theorem include two coefficients which are dependent on the kernel function. After defining those, we will finally state the theorem.

\begin{definition}
Define
\begin{equation*}
    \mathcal{C_B}:=\frac{2\left(\Kpz\right)^2-2\Kpe\Kpd}{\Kpz-2\left(\Kpe\right)^2}
\end{equation*}
and
\begin{equation*}
    \mathcal{C_S}:=\frac{\left(K^2\right)_+^{(0)}\left(\Kpz\right)^2+\left(K^2\right)^{(2)}_+\left(\Kpe\right)^2-2\left(K^2\right)^{(1)}_+\Kpz\Kpe}{\left[\left(\Kpe\right)^2-\frac{1}{2}\Kpz\right]^2}.
\end{equation*}
\end{definition}

\begin{theorem}\label{mainthm}
Suppose that the assumptions of Section \ref{secAssump} hold. There are sequences \Bn and \Sn with
\begin{equation}\label{blue1}
    \begin{split}
        \Bn = \frac{\Cb}{2}\left(\muytppa-\muytppm\right)+o(1),\quad\Sn^2 = \frac{\Cs}{f_X(0)}\left(\sigyta+\sigytm\right)+o(1)
    \end{split}
\end{equation}
such that the following statement is true:
\begin{equation*}
    \begin{split}
        \frac{\sqrt{nh}(\tauh-\tauy-h^2\Bn)}{\Sn}\xrightarrow{d}\mathcal{N}(0,1).
    \end{split}
\end{equation*}
\end{theorem}

The proof of the theorem consists of two main steps. The first one is to deduce a representation of the bias by using the result of Chapter \ref{biasrepresentation}. In the second main step, we show the asymptotic normality in the covariate-adjusted case and derive a representation of the variance, which is done in Chapter \ref{convergence}. In addition, just some extra effort is required in order to exactly obtain the representation of the bias and variance as claimed above, which is covered by the calculations in Chapters \ref{biasConversion} and \ref{varianceConversion}.

\begin{proof}[Proof of Theorem \ref{mainthm}]
We start by showing that it is enough to consider the case when $\mu_{Z}$ is continuously differentiable. Using the differentiablity and continuity statements about $\mu_Z$ in Assumption \ref{assumption3}, it can be verified that $\mu_{\tilde Z}$ is continuous and three times one-sided differentiable at zero. Also, we can calculate that $\mu_{\tilde Z}(0)=0$ and $\mu'_{\tilde Z+}=\mu'_{\tilde Z-}=0$. In particular, $\mu'_{\tilde Z}$ can be continuously extended in zero. Similarly we argue that $\mu_{\tilde Z \tilde Z^\top}$ and $\mu_{\tilde Z Y}$ are one time one-sided differentiable at zero.

By defining an estimator which is based on $V_i$ and \Zit, namely
\begin{align*}
    \left(\thetancheck(h),\gammancheck(h)\right)=\argmin_{(\theta, \gamma)\in\mathbb{R}^{p+4}} \sum_{i=1}^n K_h(X_i)\left(Y_i-V_i^\top\theta-\Zit^\top\gamma\right)^2,
\end{align*}
we can compare it to our normal estimator based on $V_i$ and $Z_i$:
\begin{align*}
    \big(\thetanhat(h),\gammanhat(h)\big)&=\argmin_{(\theta, \gamma)\in\mathbb{R}^{p+4}} \sum_{i=1}^n K_h(X_i)\left(Y_i-V_i^\top\theta-Z_i^\top\gamma\right)^2\\
    &=\argmin_{(\theta, \gamma)\in\mathbb{R}^{p+4}} \sum_{i=1}^n K_h(X_i)\left(Y_i-V_i^\top(\theta+M_n\gamma)-\Zit^\top\gamma\right)^2.
\end{align*}
In case of uniqueness of those minimizing parameters, they obviously have to be identical and we can conclude:
\begin{equation*}
    \gammancheck(h)=\gammanhat(h)\text{\quad and\quad}\thetancheck(h)=\thetanhat(h)+M_n\gammanhat(h).
\end{equation*}
If multiple minimizing parameters exist, we just state that there are parameters such that the above identification holds.
In an analogous way, we also have
\begin{equation*}
    \left(\thetazerocheck(h), \gammazerocheck(h)\right) := \argmin_{(\theta, \gamma)\in\mathbb{R}^{p+4}} \E\left(K_h(X_i)\left(Y_i-V_i^\top\theta-\Zit^\top\gamma\right)^2\right)
\end{equation*}
and
\begin{equation*}
    \left(\theta_0(h), \gamma_0(h)\right) = \argmin_{(\theta, \gamma)\in\mathbb{R}^{p+4}} \E\left(K_h(X_i)\left(Y_i-V_i^\top(\theta+M_n\gamma)-\Zit^\top\gamma\right)^2\right).
\end{equation*}
By providing the same argument as above, we obtain
\begin{equation*}
    \gammazerocheck(h)=\gamma_0(h) \text{\quad and\quad} \thetazerocheck(h)=\theta_0(h)+M_n\gamma_0(h).
\end{equation*}
As we are only interested in the effect of the treatment variable, we proceed by calculating the second entry of those parameters:
\begin{equation*}
    \thetancheck^{(2)}(h)=\thetanhat^{(2)}(h)+(M_n)_2\gammanhat(h)=\thetanhat^{(2)}(h),
\end{equation*}
whereby the second equality follows from the fact that the second row of $M_n$ is zero by definition. We can now substitute the estimators in the expression which we want to examine asymptotically:
\begin{equation}\label{preveq}
    \sqrt{\frac{nh}{\Sn^2}}\left(\thetanhat^{(2)}(h)-\thetazerocheck^{(2)}(h)\right) = \sqrt{\frac{nh}{\Sn^2}}\left(\thetancheck^{(2)}(h)-\thetazerocheck^{(2)}(h)\right).
\end{equation}
We want to apply Proposition \ref{ref1} with $Z_i=\Zit$ now. Indeed, all necessary conditions are satisfied due to Assumption \ref{assumption1} to \ref{assumption5} and the above regularity considerations on $\mu_{\tilde Z}$. Indeed, Proposition \ref{ref1} gives us $\thetazerocheck^{(2)}=\tauy+h^2\check{\mathcal{B}}_n$ whereby
\begin{equation*}
    \check{\mathcal{B}}_n=\frac{\mathcal{C_B}}{2}\left(\muyppa-\muyppm- \left(\mu''_{\tilde Z+}-\mu''_{\tilde Z-}\right)^\top\check\beta_n \right)+o(1)
\end{equation*}
with
\begin{equation*}
    \check\beta_n=\E\left(K_h(X_i)\Zit\tilde Z_i^{\top}\right)^{-1}\E\left(K_h(X_i)\Zit Y_i\right).
\end{equation*}
By Proposition \ref{lemmaE}, which converts the representation of the bias, we can conclude $\check{\mathcal{B}}_n=\Bn$, leading to
\begin{equation*}
    \thetazerocheck^{(2)}=\tauy+h^2\Bn.
\end{equation*}
Also, we know that $\thetanhat^{(2)}=\tauh$ by definition. Inserting these expressions in (\ref{preveq}) shows that it is enough to examine the asymptotic behavior of
\begin{equation*}
    \sqrt{\frac{nh}{\Sn^2}}\left(\thetancheck^{(2)}(h)-\thetazerocheck^{(2)}(h)\right)
\end{equation*}
in order to prove the assertion of the theorem.

We have
\begin{equation}\label{rcheckeqr}
\begin{split}
    \richeck(h)&=Y_i-V_i^\top\thetazerocheck(h)-\Zit^\top\gammazerocheck(h)\\
    &=Y_i-V_i^\top\big(\theta_0(h)+M_n\gamma_0(h)\big)-\left(Z_i-M_n^\top V_i\right)^\top\gamma_0(h)\\
    &=Y_i-V_i^\top\theta_0(h)-Z_i\gamma_0(h)\\
    &=r_i(h).
\end{split}
\end{equation}
As a consequence, we can replace $r_i(h)$ by $\check r_i(h)$ in Assumption \ref{assumption6} and \ref{assumption7} and they still hold. Thus, Proposition \ref{theoremA1} can be applied with $Z_i=\Zit$ since its assumptions are satisfied by using Assumption \ref{assumption1} to \ref{assumption7} from Section \ref{secAssump}. This gives us that
\begin{equation*}
    \sqrt{\frac{nh}{\check{\mathcal{S}_n^2}}}\left(\thetancheck^{(2)}(h)-\thetazerocheck^{(2)}(h)\right)\xrightarrow{d}\mathcal{N}(0,1)
\end{equation*}
whereby
\begin{equation*}
    \check{\mathcal{S}_n^2}=\frac{1}{h}\E\left(K\left(\frac{X_i}{h}\right)^2(w^\top V_i)^2\check r_i(h)^2\right).
\end{equation*}
From Proposition \ref{lemmaVar}, which converts the representation of the variance, we know that $\check{\mathcal{S}_n^2}=\Sn^2$, directly leading to
\begin{equation*}
    \frac{\sqrt{nh}(\tauh-\tauy-h^2\Bn)}{\Sn}\xrightarrow{d}\mathcal{N}(0,1)
\end{equation*}
by the above considerations. This proves the assertion of the theorem.
\end{proof}

\subsection{Consequences}
In the following, we consider the three main implications that our approach of proving the above theorem has. In particular, we discuss whether additional covariates actually enhance our estimator and its MSE. Therefore, similarly to \cite{RDDWithCov} and \cite{RDWithPotManyCov}, we want to study the bias and variance in comparison to those of the baseline estimator without covariates. Moreover, we conduct a sensitivity analysis in case of a potentially manipulated forcing variable.

\subsubsection{Discussion of Variance}
We want to show that the variance of the covariate-adjusted estimator is less or equal, but does not exceed the variance of the baseline estimator under the condition that the left- and right-sided limits of the covariance matrix of $Z_i$ are invertible, i.e. we need the additional assumption that
\begin{equation*}
    \lim_{x\searrow 0}\left(\mathrm{Cov}\left(Z_i^{(k)},Z_i^{(l)}\mid X_i=x\right)\right)_{k,l\in\{1,...,p\}}\text{ and }\lim_{x\nearrow 0}\left(\mathrm{Cov}\left(Z_i^{(k)},Z_i^{(l)}\mid X_i=x\right)\right)_{k,l\in\{1,...,p\}}
\end{equation*}
are invertible. With this condition being satisfied, the idea of the proof is to show that $\tilde\gamma$ minimizes the function
\begin{equation*}
    \gamma\mapsto\lim_{x\searrow 0}\mathrm{Var}\left(Y_i-Z_i^\top\gamma\mid X_i=x\right)+\lim_{x\nearrow 0}\mathrm{Var}\left(Y_i-Z_i^\top\gamma\mid X_i=x\right)
\end{equation*}
leading to
\begin{equation}\label{variancesmaller}
    \sigyta+\sigytm\leq\sigma^2_{Y+}+\sigma^2_{Y-}
\end{equation}
whereby the expression on the right-hand side is the leading term of the baseline estimator's variance.

This particularly shows that the variance does not worsen in comparison to the case without covariates, and is a stronger statement than the one deduced by \cite{RDDWithCov}. In fact, they state that in their setting, (\ref{variancesmaller}) is just satisfied under the condition $\tilde\gamma=\gamma_{Y+}=\gamma_{Y-}$, whereby $\gamma_{Y+}=\left(\sigma_{Z+}^2\right)^{-1}\sigma_{ZY+}^2$ and $\gamma_{Y-}=\left(\sigma_{Z-}^2\right)^{-1}\sigma_{ZY-}^2$. Yet, there are examples of RDDs without this condition being satisfied. For instance, we can assume
\begin{itemize}
    \item $\mu_{Z+}=\mu_{Z-}$, $\mu_{ZZ^\top+}=\mu_{ZZ^\top-}$,
    \item $p,q$ polynomials with $p(0)\neq q(0)$, $0\neq w\in\mathbb{R}^{p}$ and
    \item $Y_i=T_i(p(X_i)+Z_i^\top 2w)+(1-T_i)(q(X_i)+Z_i^\top w)$.
\end{itemize}
Then $\tilde\gamma=\frac{3}{2}w\neq 2w=\gamma_{Y+}$, which can be verified by calculating $\sigma_{Z+}^2=\sigma_{Z-}^2$, $\sigma_{ZY+}^2=2\sigma_{Z+}^2w$ and $\sigma_{ZY-}^2=\sigma_{Z+}^2w$. Despite this condition not being satisfied, $\mathrm{Var}(\tilde Y+)+\mathrm{Var}(\tilde Y-)\leq\mathrm{Var}(Y+)+\mathrm{Var}(Y-)$ still holds true, which is a consequence of $\sigma_{\tilde Y+}^2+\sigma_{\tilde Y-}^2=\sigma_{Y+}^2+\sigma_{Y-}^2-\frac{9}{2}w^\top\sigma_{Z+}^2w$ and $\sigma_{Z+}^2$ being positive semidefinite.

Next, we show that $\tilde\gamma$ minimizes the variance under the above invertibility assumption, nonetheless irrespective of the condition $\tilde\gamma=\gamma_{Y+}=\gamma_{Y-}$.
\begin{proof}[Proof (that $\tilde\gamma$ minimizes the function).]
We have
\begin{align*}
    &\nabla_\gamma\left(\lim_{x\searrow 0}\mathrm{Var}\left(Y_i-Z_i^\top\gamma\mid X_i=x\right)+\lim_{x\nearrow 0}\mathrm{Var}\left(Y_i-Z_i^\top\gamma\mid X_i=x\right)\right)\\
    =\;&2\left(\left(\sigma^2_{Z-}+\sigma^2_{Z+}\right)\gamma-\left(\sigma^2_{ZY-}+\sigma^2_{ZY+}\right)\right).
\end{align*}
This gradient vanishes when inserting $\tilde\gamma$. Now, it remains to show that the Hessian matrix is positive definite. In fact, we obtain
\begin{equation}\label{hessian}
\begin{split}
    &\mathrm{H}_{\gamma\mapsto\left(\lim_{x\searrow 0}\mathrm{Var}\left(Y_i-Z_i^\top\gamma\mid X_i=x\right)+\lim_{x\nearrow 0}\mathrm{Var}\left(Y_i-Z_i^\top\gamma\mid X_i=x\right)\right)}\\
    \equiv\;&2\lim_{x\searrow 0}\left(\mathrm{Cov}\left(Z_i^{(k)},Z_i^{(l)}\mid X_i=x\right)\right)_{k,l\in\{1,...,p\}}\\
    &+2\lim_{x\nearrow 0}\left(\mathrm{Cov}\left(Z_i^{(k)},Z_i^{(l)}\mid X_i=x\right)\right)_{k,l\in\{1,...,p\}}.
\end{split}
\end{equation}
We consider the first summand, and obtain for $y\in\mathbb{R}^p\setminus\{0\}$ that
\begin{align*}
    &y^\top\lim_{x\searrow 0}\left(\mathrm{Cov}\left(Z_i^{(k)},Z_i^{(l)}\mid X_i=x\right)\right)_{k,l\in\{1,...,p\}} y\\
    =\;&\lim_{x\searrow 0}\E\left(\left(\sum_{k=1}^p y_k\left(Z_i^{(k)}-\E\left(Z_i^{(k)}\mid X_i=x\right)\right)\right)^2\;\,\vrule\; X_i=x\right)\overset{(*)}{\geq} 0.
\end{align*}
Hence, the matrix $\lim_{x\searrow 0}\left(\mathrm{Cov}\left(Z_i^{(k)},Z_i^{(l)}\mid X_i=x\right)\right)_{k,l\in\{1,...,p\}}$ is positive semidefinite. Consequently, we can apply the Cholesky decomposition to get
\begin{equation*}
    \lim_{x\searrow 0}\left(\mathrm{Cov}\left(Z_i^{(k)},Z_i^{(l)}\mid X_i=x\right)\right)_{k,l\in\{1,...,p\}}=L^\top DL
\end{equation*}
for some orthogonal lower triangular matrix $L\in\mathbb{R}^{p\times p}$ and a diagonal matrix $D\in\mathbb{R}^{p\times p}$ having the real eigenvalues on its diagonal. Note that those eigenvalues are non-negative since the matrix is positive semidefinite and unequal to zero since we assumed the limit of the covariance matrix to be invertible. For this reason, $D$ as well as $D^{\frac{1}{2}}$ are invertible, whereby $D^{\frac{1}{2}}$ is the diagonal matrix obtained by taking the square root of the diagonal entries of $D$. Therefore, $D^{\frac{1}{2}}L$ is invertible, which means that
\begin{equation*}
    y^\top L^\top DL y=y^\top L^\top D^{\frac{1}{2}}D^{\frac{1}{2}} L y=(D^{\frac{1}{2}}Ly)^\top D^{\frac{1}{2}}Ly=\norm{D^{\frac{1}{2}}Ly}_2^2=0\quad\Leftrightarrow\quad y=0.
\end{equation*}
Therefore, even a strict inequality holds in $(*)$, which proves that the matrix is positive definite. In an analogous way, we can prove that also the second summand in (\ref{hessian}) is positive definite, leading to the whole Hessian being positive definite. Finally, we can conclude that $\tilde\gamma$ is a minimum.

\end{proof}

\subsubsection{Discussion of Bias and Higher Order Estimation}
The first term of the bias representation can be written as
\begin{equation*}
    \muytppa-\muytppm=\left(\muyppa-\muyppm\right)-\left(\mu''_{Z+}-\mu''_{Z-}\right)^\top\tilde\gamma.
\end{equation*}

The first summand equals the bias of the baseline estimator without covariates up to a constant, whereby the second summand describes the impact of the covariates on the bias. And indeed, the covariates can contribute to the bias in both a favorable and unfavorable way. If researchers are concerned that additional covariates introduce additional bias but if they are generally happy with a convergence rate of $n^{-2/5}$ (corresponding to $h=n^{-1/5}$), we demonstrate in the following slightly unconventional analysis that our proof can be extended to higher order locally polynomial estimators by keeping the same order of Taylor expansions. More precisely, define
\begin{equation*}
    V_i=\left(1, T_i, \frac{X_i}{h}, \frac{T_iX_i}{h}, \frac{X_i^2}{h^2}, \frac{T_iX_i^2}{h^2}\right)^\top
\end{equation*}
and
\begin{equation*}
    M_n=\left(\mu_Z(0), 0, h\mu_{Z-}', h(\mu_{Z+}'-\mu_{Z-}'), h^2\mu_{Z-}'', h^2(\mu_{Z+}''-\mu_{Z-}'')\right)^\top,
\end{equation*}
and replace it in all expressions of Section \ref{setup} and in the assumptions of Section \ref{secAssump}. That is, instead of local linear regression, we perform local quadratic regression. Then, we obtain $\mu_{\tilde Z}(0)=\mu_{\tilde Z+}'=\mu_{\tilde Z-}'=\mu_{\tilde Z+}''=\mu_{\tilde Z-}''=0$.
That is, we additionally correct for the jump of $\mu_{Z}''$ at zero. While analyzing this estimator, we do not want to impose further assumptions on the derivatives of $f_X$ and $\mu_Y$.

We want to provide some details on how to adjust the proof: In the same manner as before, we obtain statement (\ref{preveq}). Now, the statements in Section \ref{biasrepresentation} can be adjusted in a natural way such that they can be applied on (the new versions of) $\Zit$ and $V_i$, whereby we have to make a slightly stronger assumption of $K^{(6)}<\infty$. In fact, the adjusted version of Lemma \ref{lemmaA12} gives
\begin{equation*}
    \kappa(K)^{-1}\E(K_h(X_i)V_iA)=\underbrace{\begin{pmatrix}
    f_X(0)\mu_{A-}\\f_X(0)\tau_A\\h[\mu_Af_X]'_-\\h([\mu_Af_X]'_+-[\mu_Af_X]'_-)\\0\\0
    \end{pmatrix}}_{(*)}+h^2B(K,A)+O(h^3)
\end{equation*}
for $h\to 0$, whereby
\begin{equation*}
    B(K,A):=\frac{1}{2}\kappa(K)^{-1}\left[ \begin{pmatrix}
    \Kpz\\\Kpz\\\Kpd\\\Kpd\\K_+^{(4)}\\K_+^{(4)}
    \end{pmatrix}[\mu_Af_X]''_+ + \begin{pmatrix}
    \Kmz\\0\\\Kmd\\0\\K_-^{(4)}\\0
    \end{pmatrix}[\mu_Af_X]''_- \right]
\end{equation*}
and
\begin{equation*}
    \kappa(K):=\begin{pmatrix}
    \Kn&\Kpn&\Ke&\Kpe&\Kz&\Kpz\\\
    \Kpn&\Kpn&\Kpe&\Kpe&\Kpz&\Kpz\\
    \Ke&\Kpe&\Kz&\Kpz&\Kd&\Kpd\\
    \Kpe&\Kpe&\Kpz&\Kpz&\Kpd&\Kpd\\
    \Kz&\Kpz&\Kd&\Kpd&K^{(4)}&K_+^{(4)}\\
    \Kpz&\Kpz&\Kpd&\Kpd&K_+^{(4)}&K_+^{(4)}
    \end{pmatrix}.
\end{equation*}
The relevant equations for the bias formula in the proof of Proposition \ref{ref1} are \eqref{eq:bias_Y} and \eqref{eq:bias_Z}. Since we may assume twice differentiability of $\mu_{\tilde{Z}}$, we conclude that $B(K,\tilde{Z})=0$. Therefore, \eqref{eq:bias_Z} provides no contribution to the bias of order $h^2$. Note furthermore that in this case
$$B(K,Y)=\frac{1}{2}\left[\begin{pmatrix}0 \\ 0 \\ 0 \\ 0 \\ 0 \\ 1\end{pmatrix}\left[\mu_Yf_X\right]_+''+\begin{pmatrix} 0 \\ 0 \\ 0 \\ 0 \\ 1 \\ -1 \end{pmatrix}\left[\mu_Yf_X\right]_-''\right]=\frac{1}{2}\begin{pmatrix} 0 \\0 \\0 \\0 \\ \left[\mu_Yf_X\right]_-'' \\ \left[\mu_Yf_X\right]_+''-\left[\mu_Yf_X\right]_-'' \\\end{pmatrix}.$$
In \eqref{eq:bias_Y} we see that the second entry of the previous vector is relevant for the bias of order $h^2$, i.e., in this case it does not exist. Moreover, for the second part of \eqref{eq:bias_Y} it can be shown that
$$\kappa_{h,b}=\begin{pmatrix} O(h^3) & O(h^3) & O(h^2) & O(h^2) & O(h) & O(h) \end{pmatrix}$$
in equation \eqref{kappahb}, implying that there is only a bias of order $o(h^2)$. Thus, in a local quadratic regression, we obtain a representation of the bias which does not depend on the influence of the covariates anymore.

\begin{subsubsection}{Sensitivity Analysis}
In this section we show how Theorem \ref{mainthm} can be applied in case $\tau_Y$ does not admit a causal interpretation. Examples for confounding through manipulation of the forcing variable $X_i$ are given in, e.g., \cite{GRR20}. In that paper, the manipulation is reflected as a potential discontinuity of the density of the forcing variable by dividing the units in two different categories: so called \textit{always-assigned} units, which always have a realization of the forcing variable greater or equal the cutoff (i.e. we have a one-sided manipulation), and \textit{potentially-assigned} units, which behave according to the standard RD framework. Let us therefore suppose that there are potential outcomes $Y_i(t)$ with continuous $\IE(Y_i(t)\mid X_i=x)$ for $t\in\{0,1\}$ but that we observe
$$Y_i=T_iY_i(1)+(1-T_i)Y_i(0)+\epsilon_i,$$
where we suppose that $\Delta:=\lim_{x\searrow 0} \IE(\epsilon\mid X_i=x)-\lim_{x\nearrow 0} \IE(\epsilon_i\mid X_i=x)\neq0$. Denote furthermore $\tau_0:=\IE(Y_i(1)\mid X_i=0)-\IE(Y_i(0)\mid X_i=0)$. In this case Theorem \ref{mainthm} still applies with $\tau_Y=\tau_0+\Delta$. We call $\tau_0$ the effect of interest and $\Delta$ the level of confounding. Suppose our interest lies in understanding if $\tau_0>\overline{\tau}$ for some given $\overline{\tau}>0$. While it is not possible to test this hypothesis directly, Theorem \ref{mainthm} still allows to conduct a sensitivity analysis as in \cite{JRC23}. Intuitively speaking, the idea is to test for statements of the type, either $\tau_0>\overline{\tau}$ or the confounding level is larger than $\hat{\delta}_n$. If the confounding level $\hat{\delta}_n$ is unplausibly high, one would have to accept that the effect $\tau_0$ is larger than $\overline{\tau}$.  More formally, consider for $\delta>0$ the hypotheses
$$H_0(\delta):\quad \tau_0\leq\overline{\tau}\textrm{ and }|\Delta|\leq\delta.$$
These hypotheses formalize the previously mentioned intuition: we reject $H_0(\delta)$ at level $\alpha\in(0,1)$ either if the true effect is larger than $\overline{\tau}$ or if the level of confounding is larger than $\delta$. Our interest is finding the largest $\delta$ for which $H_0(\delta)$ would be rejected. A natural strategy is to test all hypotheses $H_0(\delta)$ and let $\mathcal{R}:=\{\delta>0:\,H_0(\delta)\textrm{ rejected}\}$. Then, $\hat{\delta}_n:=\sup\mathcal{R}$ denotes the largest hypothesis which would be rejected on confidence level $\alpha$. \cite{JRC23} argue that this strategy does not run into multiple testing issues. We explain their argument in our setting.

First, we specify confidence intervals for $\tau_0$ under the assumption that the level of confounding is smaller than $\delta>0$. Let $\hat{\Sn}$ be an estimator for $\Sn$ such that $\hat{\Sn}/\Sn\overset{P}{\to}1$ and suppose undersmoothing, i.e., $nh^5\to0$. In this case we obtain from Theorem \ref{mainthm} that
$$\frac{\sqrt{nh}}{\hat{\Sn}}\left(\hat{\tau}_h-\tau_0-\Delta\right)\overset{d}{\to}\mathcal{N}(0,1).$$
As an alternative to undersmoothing, one could certainly handle the bias differently as, e.g., in \cite{RDDWithCov} or \cite{Armstrong_2018}, but we would like to present an approach that is compatible with Theorem \ref{mainthm}. Let $q_{1-\alpha}$ denote the $(1-\alpha)$-quantile of the standard normal distribution. Then, we obtain under the assumption $|\Delta|\leq\delta$ that
\begin{align*}
&\limsup_{n\to\infty}\IP\left(\tau_0\geq\hat{\tau}_h-\delta-\frac{\hat{\Sn}}{\sqrt{nh}}q_{1-\alpha}\right)\geq\limsup_{n\to\infty}\IP\left(\tau_0\geq\hat{\tau}_h-\Delta-\frac{\hat{\Sn}}{\sqrt{nh}}q_{1-\alpha}\right) \\
=&\limsup_{n\to\infty}\IP\left(\frac{\sqrt{nh}}{\hat{\Sn}}\left(\hat{\tau}_h-\tau_0-\Delta\right)\leq q_{1-\alpha}\right)=1-\alpha.
\end{align*}
Therefore, $C_n(\delta):=(\hat{\tau}_h-\delta-q_{1-\alpha}\hat{\Sn}/\sqrt{nh},\infty)$ is a confidence interval for $\tau_0$ of level $1-\alpha$ under the assumption $|\Delta|\leq\delta$. This directly yields a way of testing each hypothesis $H_0(\delta)$ at level $\alpha$: reject $H_0(\delta)$ if $\overline{\tau}\notin C_n(\delta)$. Using this test, we obtain,
$$\mathcal{R}:=\left(0,\hat{\tau}_h-\overline{\tau}-q_{1-\alpha}\frac{\hat{\Sn}}{\sqrt{nh}}\right].$$
Let furthermore $\mathcal{H}_0$ be the set of all true hypotheses, i.e., if $\tau_0>\overline{\tau}$, then $\mathcal{H}_0=\emptyset$ and in case $\tau_0\leq\overline{\tau}$, we have $\mathcal{H}_0=[|\Delta|,\infty)$. We can show (as in \cite{JRC23}) that in case of no true effect, i.e. $\tau_0\leq\overline{\tau}$, we reject no $H_0(\delta)$ with an unreasonably large $\delta$. In fact, the probability of rejecting a true hypothesis is bounded from above by $\alpha$: if $\tau_0\leq\overline{\tau}$, we get
\begin{align*}
\IP(\mathcal{R}\cap\mathcal{H}_0\neq\emptyset)=&\IP\left(\hat{\tau}_h-\overline{\tau}-q_{1-\alpha}\frac{\hat{\Sn}}{\sqrt{nh}}\geq|\Delta|\right)\leq\IP\left(\hat{\tau}_h-\tau_0-q_{1-\alpha}\frac{\hat{\Sn}}{\sqrt{nh}}\geq\Delta\right) \\
=&\IP\left(\frac{\sqrt{nh}}{\hat{\Sn}}\left(\hat{\tau}_h-\tau_0-\Delta\right)\geq q_{1-\alpha}\right)\to\alpha.
\end{align*}
Hence, we conclude that, if there is no large effect $\tau_0>\overline{\tau}$, we reject only with probability $\alpha$ a hypothesis which is in fact true. Therefore we may say that if there is no effect larger than $\overline{\tau}$, then we have to accept a level of confounding
$$|\Delta|\geq \hat{\delta}_n:=\hat{\tau}_h-\overline{\tau}-q_{1-\alpha}\frac{\hat{\Sn}}{\sqrt{nh}}.$$
\end{subsubsection}

\section{Statements Used in the Theorem}\label{proofsection}
This section contains all statements for the above main theorem's proof. In fact, we have to show four propositions, whereby each is stated together with its proof separately in one of the following subsections. Supporting lemmas which are used can be found in the appendix of this paper.

\begin{subsection}{Computing the Representation of the Bias}\label{biasrepresentation}

The aim of this section is to provide a representation of the bias such that we have $\theta_0^{(2)} = \tauy + h^2\Bn$. This will be achieved by using a kernel matrix and performing numerous Taylor expansions. Some ideas originate from \cite{RDWithPotManyCov} and are adapted to the case of a fixed number of covariates.

\begin{proposition}\label{ref1}
Let all assumptions of Lemma \ref{lemmaA11} and Lemma \ref{lemmaA13} hold. Also, suppose that $\mu_Y$ is three times one-sided differentiable in zero (cf. Definition \ref{defofonesideddiff}). Furthermore, define
\begin{align*}
    \Bn := &\;\frac{(\Kpz)^2-\Kpe\Kpd}{\Kpz-2(\Kpe)^2}\Bigg( \muyppa-\muyppm-\left(\mu''_{Z+}-\mu''_{Z-}\right)^\top\E(K_h(X_i)Z_iZ_i^\top)^{-1}\E(K_h(X_i)Z_iY_i) \Bigg)\\
    &+ o(1).
\end{align*}
Then,
\begin{equation*}
    \theta_0^{(2)} = \tauy + h^2\Bn.
\end{equation*}
\end{proposition}
\begin{proof}
Choose $a_1,a_2,b_1,b_2$ like in Lemma \ref{lemmaA14} and set
\begin{align*}
    \kappa_{h,b}(K):=&\left[ \left( I+h\frac{f_X'(0)}{f_X(0)} \begin{pmatrix}0&0&a_1&0\\0&0&0&a_1\\1&0&-a_2&0\\0&1&2a_2&a_2\end{pmatrix} +h^2\frac{f_X''(0)}{2f_X(0)} \begin{pmatrix}a_1&0&-b_1&0\\0&a_1&2b_1&b_1\\-a_2&0&b_2&0\\2a_2&a_2&0&b_2\end{pmatrix} \right)^{-1} \right]_2-\begin{pmatrix}0&1&0&0\end{pmatrix}.
\end{align*}
We know by definition that
\begin{equation*}
    (\theta_0(h), \gamma_0(h))=\argmin_{(\theta,\gamma)\in\mathbb{R}^{p+4}}\E\left(K_h(X_i)(Y_i-\theta^\top V_i-\gamma^\top Z_i)^2\right).
\end{equation*}
By using the least squares algebra, we obtain
\begin{align*}
    \begin{pmatrix}\theta_0(h)\\ \gamma_0(h)\end{pmatrix}=\begin{pmatrix}\E(K_h(X_i)V_iV_i^\top)&\E(K_h(X_i)V_iZ_i^\top)\\\E(K_h(X_i)Z_iV_i^\top)&\E(K_h(X_i)Z_iZ_i^\top)\end{pmatrix}^{-1} \begin{pmatrix}\E(K_h(X_i)V_iY_i)\\\E(K_h(X_i)Z_iY_i)\end{pmatrix}
\end{align*}
We calculate the inverse of the matrix by using the well-known formula for inverses of block matrices. Since our goal is to find a formula for $\theta_0^{(2)}$, we only need to calculate the upper blocks of the inverse.
\begin{align*}
&\begin{pmatrix}\E(K_h(X_i)V_iV_i^\top)&\E(K_h(X_i)V_iZ_i^\top)\\\E(K_h(X_i)Z_iV_i^\top)&\E(K_h(X_i)Z_iZ_i^\top)\end{pmatrix}^{-1}=
    \begin{pmatrix}
    R & -R\E(K_h(X_i)V_iZ_i^\top)\E(K_h(X_i)Z_iZ_i^\top)^{-1} \\ * & *
    \end{pmatrix}
\end{align*}
whereby $R:=(\E(K_h(X_i)V_iV_i^\top)-\E(K_h(X_i)V_iZ_i^\top)\E(K_h(X_i)Z_iZ_i^\top)^{-1}\E(K_h(X_i)Z_iV_i^\top))^{-1}$.
Denote the second row of $R$ by $R_2$. Now, the second entry of $\theta_0(h)$ can be represented as follows:
\begin{equation}\label{repofthetazerocheck}
\begin{split}
    \theta_0^{(2)}(h)=&\bigg[\Big(I-\E\left(K_h(X_i)V_iV_i^\top\right)^{-1}\E\left(K_h(X_i)V_iZ_i^\top\right)\E\left(K_h(X_i)Z_iZ_i^\top\right)^{-1}\\
    &\E\left(K_h(X_i)Z_iV_i^\top\right)\Big)^{-1}\bigg]_2\left(\kappa(K)^{-1}\E\left(K_h(X_i)V_iV_i^\top\right)\right)^{-1}\kappa(K)^{-1}\\
    &\Big(\E(K_h(X_i)V_iY_i)-\E\left(K_h(X_i)V_iZ_i^\top\right)\E\left(K_h(X_i)Z_iZ_i^\top\right)^{-1}\E(K_h(X_i)Z_iY_i)\Big).
\end{split}
\end{equation}
By applying Lemma \ref{lemmaA13} and Lemma \ref{lemmaA15} we obtain that
\begin{align*}
    &\Bigg[\bigg(I-\E\left(K_h(X_i)V_iV_i^\top\right)^{-1}\E\left(K_h(X_i)V_iZ_i^\top\right)\E\left(K_h(X_i)Z_iZ_i^\top\right)^{-1}\E\left(K_h(X_i)Z_iV_i^\top\right)\bigg)^{-1}\Bigg]_2\\
    &\left(\kappa(K)^{-1}\E\left(K_h(X_i)V_iV_i^\top\right)\right)^{-1}\\
    =\;&\frac{1}{f_X(0)}\left( \begin{pmatrix}0&1&0&0\end{pmatrix}+\kappa_{h,b}(K) \right)+o(h^2).
\end{align*}
We set $A=Y_i$ in Lemma \ref{lemmaA12} and obtain
\begin{equation*}
    \kappa(K)^{-1}\E(K_h(X_i)V_iY_i)=\begin{pmatrix}
    f_X(0)\mu_{Y-}\\f_X(0)\tauy\\h(\mu_Yf_X)'_{-}\\h((\mu_Yf_X)'_{+}-(\mu_Yf_X)'_{-})
    \end{pmatrix}+h^2B(K,Y)+O(h^3).
\end{equation*}
Also, we can apply this lemma with $A=Z_i^{(k)}$ for $1\leq k\leq p$:
\begin{equation}\label{taylorOfVZ}
\begin{split}
    &\kappa(K)^{-1}\E(K_h(X_i)V_iZ_i^\top)\\
    =\;&\kappa(K)^{-1}\E\left(K_h(X_i)V_i\left(\sum_{k=1}^p Z_i^{(k)}e_k^\top\right)\right)\\
    =\;&\sum_{k=1}^p \kappa(K)^{-1}\E\left(K_h(X_i)V_iZ_i^{(k)}\right)e_k^\top\\
    =\;&\sum_{k=1}^p\left(\begin{pmatrix}
    f_X(0)\mu_{Z^{(k)}-}\\f_X(0)\tau_{Z^{(k)}}\\h(\mu_{Z^{(k)}}f_X)'_{-}\\h((\mu_{Z^{(k)}}f_X)'_{+}-(\mu_{Z^{(k)}}f_X)'_{-})
    \end{pmatrix}+h^2B\left(K,Z_i^{(k)}\right)+O(h^3)\right)e_k^\top\\
    =\;&\sum_{k=1}^p \left(h^2B\left(K,Z_i^{(k)}\right)+O(h^3)\right)e_k^\top\\
    =\;&h^2B+O(h^3)
\end{split}
\end{equation}
whereby $B$ is defined as the matrix which has $B(K,Z_i^{(k)}), k=1,...,p$ as columns. Note that we used the assumptions $\mu_{Z^{(k)}}(0)=0$ as well as $\lim_{x\searrow 0} \mu'_{Z^{(k)}}(x)=\lim_{x\nearrow 0} \mu'_{Z^{(k)}}(x)=0$, which also implies $\tau_{Z^{(k)}}=0$. Now, we can state that
\begin{align*}
    &\kappa(K)^{-1}\E(K_h(X_i)V_iZ_i^\top)\E(K_h(X_i)Z_iZ_i^\top)^{-1}\E(K_h(X_i)Z_iY_i)\\
    =\;&h^2B\E(K_h(X_i)Z_iZ_i^\top)^{-1}\E(K_h(X_i)Z_iY_i)+O(h^3)\E(K_h(X_i)Z_iZ_i^\top)^{-1}\E(K_h(X_i)Z_iY_i).
\end{align*}
We want to find an upper bound for that. This can be done by noticing that
\begin{equation*}
    \norm{\E(K_h(X_i)Z_iZ_i^\top)^{-1}\E(K_h(X_i)Z_iY_i)}_2=O(1)
\end{equation*}
by Lemma \ref{lemmaA11}. This leads to
\begin{align*}
    &\left[\kappa(K)^{-1}\E(K_h(X_i)V_iZ_i^\top)\E(K_h(X_i)Z_iZ_i^\top)^{-1}\E(K_h(X_i)Z_iY_i)\right]_l\\
    \leq\;& h^2\norm{B_{l\cdot}}_2O(1)+\norm{(O(h^3))_{l\cdot}}_2O(1)=O(h^2)+O(h^3)=O(h^2)
\end{align*}
for $l\in\{1,...,4\}$. Now we merge all of the statements above to obtain
\begingroup
\allowdisplaybreaks
\begin{align*}
    &\theta_0^{(2)}(h)\\
    =\;&\left[ \frac{1}{f_X(0)}\left(\begin{pmatrix}0&1&0&0\end{pmatrix}+\kappa_{h,b}(K)\right)+o(h^2) \right]\\
    &\cdot \kappa(K)^{-1}\left(\E(K_h(X_i)V_iY_i)-\E(K_h(X_i)V_iZ_i^\top)\E(K_h(X_i)Z_iZ_i^\top)^{-1}\E(K_h(X_i)Z_iY_i)\right)\\
    =\;&\left[ \frac{1}{f_X(0)}\left(\begin{pmatrix}0&1&0&0\end{pmatrix}+\kappa_{h,b}(K)\right)+o(h^2) \right]\cdot \left[\rule{0cm}{1.3cm}\right. \begin{pmatrix}
    f_X(0)\mu_{Y-}\\f_X(0)\tauy\\h(\mu_Yf_X)'_{-}\\h((\mu_Yf_X)'_{+}-(\mu_Yf_X)'_{-})
    \end{pmatrix}+h^2B(K,Y)+O(h^3)\\
    &-h^2B\E(K_h(X_i)Z_iZ_i^\top)^{-1}\E(K_h(X_i)Z_iY_i)-\underbrace{O(h^3)\E(K_h(X_i)Z_iZ_i^\top)^{-1}\E(K_h(X_i)Z_iY_i)}_{=O(h^3)}\left]\rule{0cm}{1.3cm}\right. \\
    =\;&\tauy+\frac{h^2}{f_X(0)}[B(K,Y)]_2-\frac{h^2}{f_X(0)}B_{2\cdot}\E(K_h(X_i)Z_iZ_i^\top)^{-1}\E(K_h(X_i)Z_iY_i) \\
    &+\underbrace{[O(h^3)]_2}_{=O(h^3)}-\frac{1}{f_X(0)}\underbrace{\kappa_{h,b}(K)}_{=o(1)}(\underbrace{h^2B\E(K_h(X_i)Z_iZ_i^\top)^{-1}\E(K_h(X_i)Z_iY_i)+O(h^3)}_{=O(h^2)})\\
    &+\kappa_{h,b}(K)\begin{pmatrix}
    \mu_{Y-}\\ \tauy\\\frac{h}{f_X(0)}(\mu_Yf_X)'_{-}\\\frac{h}{f_X(0)}((\mu_Yf_X)'_{+}-(\mu_Yf_X)'_{-})
    \end{pmatrix}+\underbrace{\frac{\kappa_{h,b}(K)}{f_X(0)}h^2B(K,Y)}_{=o(h^2)}+\underbrace{\frac{\kappa_{h,b}(K)}{f_X(0)}O(h^3)}_{=o(h^2)}\\
    &+\underbrace{o(h^2)\begin{pmatrix}
    f_X(0)\mu_{Y-}\\f_X(0)\tauy\\h(\mu_Yf_X)'_{-}\\h((\mu_Yf_X)'_{+}-(\mu_Yf_X)'_{-})
    \end{pmatrix}}_{=o(h^2)}+\underbrace{o(h^2)h^2B(K,Y)}_{=o(h^2)}+\underbrace{o(h^2)O(h^2)}_{=o(h^2)} \\
    =\;&\tauy+\frac{h^2}{f_X(0)}[B(K,Y)]_2-\frac{h^2}{f_X(0)}B_{2\cdot}\E(K_h(X_i)Z_iZ_i^\top)^{-1}\E(K_h(X_i)Z_iY_i)\\
    &+\kappa_{h,b}(K)\begin{pmatrix}
    \mu_{Y-}\\ \tauy\\\frac{h}{f_X(0)}(\mu_Yf_X)'_{-}\\\frac{h}{f_X(0)}((\mu_Yf_X)'_{+}-(\mu_Yf_X)'_{-})
    \end{pmatrix}+o(h^2),
\end{align*}
\endgroup
whereby we used that $\kappa_{h,b}(K)=o(1)$ which is shown by taking the definition and letting $h\xrightarrow{}0$. The main part of the proof is done. The rest is conducted by just following some straightforward computations. As the next calculation is rather long and technical, we want to refer to using a CAS in order to obtain
\begin{equation}\label{kappahb}
    \kappa_{h,b}(K)=\begin{pmatrix}
    o(h^2)&h^2a_1\left(\frac{f_X'(0)^2}{f_X(0)^2}-\frac{f_X''(0)}{2f_X(0)}\right)+o(h^2)&O(h^2)&-ha_1\frac{f_X'(0)}{f_X(0)}+O(h^2)
    \end{pmatrix}.
\end{equation}
We now consider the result of Lemma \ref{lemmaA14} by examining the matrix multiplication column-wise, which gives us that
\begin{equation*}
    \kappa(K)^{-1}\begin{pmatrix}\Kpz\\\Kpz\\\Kpd\\\Kpd\end{pmatrix}=\begin{pmatrix}0\\a_1\\0\\a_2\end{pmatrix}
\end{equation*}
and
\begin{align*}
    \kappa(K)^{-1}\begin{pmatrix}K_-^{(2)}\\0\\K_-^{(3)}\\0\end{pmatrix}&=\kappa(K)^{-1}\begin{pmatrix}K_+^{(2)}\\0\\-K_+^{(3)}\\0\end{pmatrix}=\kappa(K)^{-1}\left[\begin{pmatrix}
    2\Kpz\\ \Kpz\\0\\ \Kpd
    \end{pmatrix}-\begin{pmatrix}
    \Kpz\\ \Kpz \\ \Kpd\\ \Kpd
    \end{pmatrix}\right]\\
    &=\begin{pmatrix}
    a_1\\0\\-a_2\\2a_2
    \end{pmatrix}-\begin{pmatrix}
    0\\a_1\\0\\a_2
    \end{pmatrix}=\begin{pmatrix}
    a_1\\-a_1\\-a_2\\a_2
    \end{pmatrix}
\end{align*}
whereby we used that $K_-^{(2)}=\Kpz$ and $K_-^{(3)}=-\Kpd$ since the kernel is symmetric. Replacing the above identities in the definition of $B(K,Y)$ delivers
\begin{align*}
    [B(K,Y)]_2=\frac{a_1}{2}\left(\mu''_{Y+}f_X(0)+2\mu'_{Y+}f'_X(0)+\mu_{Y+}f''_X(0)-\mu''_{Y-}f_X(0)-2\mu'_{Y-}f'_X(0)-\mu_{Y-}f''_X(0)\right).
\end{align*}
Therefore
\begin{align}
    &\frac{h^2}{f_X(0)}[B(K,Y)]_2+\kappa_{h,b}(K)\begin{pmatrix}
    \mu_{Y-}\\ \tauy\\\frac{h}{f_X(0)}(\mu_Yf_X)'_{-}\\\frac{h}{f_X(0)}((\mu_Yf_X)'_{+}-(\mu_Yf_X)'_{-}) 
    \end{pmatrix} \nonumber \\
    =\;&\frac{h^2}{f_X(0)}\frac{a_1}{2}\left( f_X(0)\left(\mu''_{Y+}-\mu''_{Y-}\right)+2f_X'(0)\left(\mu'_{Y+}-\mu'_{Y-}\right)+f_X''(0)\left(\mu_{Y+}-\mu_{Y-}\right) \right) \nonumber \\
    &+\underbrace{o(h^2)\mu_{Y-}}_{=o(h^2)}+h^2a_1\left(\frac{f_X'(0)^2}{f_X(0)^2}-\frac{f_X''(0)}{2f_X(0)}\right)\left(\mu_{Y+}-\mu_{Y-}\right)+\underbrace{o(h^2)\tauy}_{=o(h^2)} \nonumber \\
    &+\underbrace{O(h^2)\frac{h}{f_X(0)}(\mu_Yf_X)'_-}_{=o(h^2)}-ha_1\frac{f_X'(0)}{f_X(0)}\frac{h}{f_X(0)}\left((\mu_Yf_X)'_+-(\mu_Yf_X)'_-\right) \nonumber \\
    &+\underbrace{O(h^2)\frac{h}{f_X(0)}((\mu_Yf_X)'_+-(\mu_Yf_X)'_-)}_{=o(h^2)} \nonumber \\
    =\;&h^2a_1\bigg(\frac{1}{2}\left(\mu''_{Y+}-\mu''_{Y-}\right)+\frac{f_X'(0)}{f_X(0)}\left(\mu'_{Y+}-\mu'_{Y-}\right)+\frac{1}{2}\frac{f_X''(0)}{f_X(0)}\left(\mu_{Y+}-\mu_{Y-}\right) \nonumber \\
    &+\frac{f_X'(0)^2}{f_X(0)^2}\left(\mu_{Y+}-\mu_{Y-}\right)-\frac{1}{2}\frac{f_X''(0)}{f_X(0)}\left(\mu_{Y+}-\mu_{Y-}\right) \nonumber \\
    &-\frac{f_X'(0)}{f_X(0)^2}\left(\mu'_{Y+}f_X(0)+\mu_{Y+}f_X'(0)-\mu'_{Y-}f_X(0)-\mu_{Y-}f'_X(0)\right) \bigg)+o(h^2) \nonumber \\
    =\;&h^2\frac{a_1}{2}\left(\mu''_{Y+}-\mu''_{Y-}\right)+o(h^2). \label{eq:bias_Y}
\end{align}
Using a representation of $B(K,Z_i^{(k)})$ in the above manner again, gives that
\begin{align}
    &\frac{h^2}{f_X(0)}B_{2\cdot}\E\left(K_h(X_i)Z_iZ_i^\top\right)^{-1}\E(K_hZ_iY_i) \nonumber \\
    =\;&h^2\frac{a_1}{2}\sum_{k=1}^p \left(\mu''_{Z^{(k)}+}-\mu''_{Z^{(k)}-}\right)\left[\E\left(K_h(X_i)Z_iZ_i^\top\right)^{-1}\E(K_hZ_iY_i)\right]^{(k)} \label{eq:bias_Z}
\end{align}
since $\mu_{Z+}=\mu_{Z-}$ and $\mu'_{Z+}=\mu'_{Z-}$. Combining the above statements proves the proposition.
\end{proof}

\end{subsection}

\begin{subsection}{Conversion of the Bias Representation}\label{biasConversion}
The next proposition covers all calculations which are necessary to transform the representation of the bias obtained in Section \ref{biasrepresentation} into the one required in the main theorem's proof.

\begin{proposition}[Computation for the bias]\label{lemmaE}
Let the assumptions of Section \ref{secAssump} hold. Then
\begin{equation*}
\mu''_{Y+}-\mu''_{Y-}-\left(\mu''_{\tilde Z+}-\mu''_{\tilde Z-}\right)^\top\check\beta_n=\mu''_{\tilde Y+}-\mu''_{\tilde Y-}+o(1)
\end{equation*}
whereby
\begin{equation*}
\check\beta_n=\E\left(K_h(X_i)\Zit\tilde Z_i^{\top}\right)^{-1}\E\left(K_h(X_i)\Zit Y_i\right).
\end{equation*}
\end{proposition}
\begin{proof}
First, we know that
\begin{equation}\label{blue2}
    \mu''_Z=\mu''_{\tilde Z}.
\end{equation}
For this reason, we can also proof
\begin{equation*}
    \mu''_{Y+}-\mu''_{Y-}-\left(\mu''_{Z+}-\mu''_{Z-}\right)^\top\check\beta_n=\mu''_{\tilde Y+}-\mu''_{\tilde Y-}+o(1)
\end{equation*}
in order to show the lemma.
Since we know by definition that for $x\neq 0$
\begin{align*}
    \mu''_{\tilde Y}(x)&=\mu''_{Y-Z^\top\tilde\gamma}(x)\\
    &=\frac{\text{d}^2}{\text{d}x^2}\,\E\left(Y_i-Z_i^\top\tilde\gamma\mid X_i=x\right)\\
    &=\frac{\text{d}^2}{\text{d}x^2}\,\E(Y_i\mid X_i=x)-\frac{\text{d}^2}{\text{d}x^2}\,\E(Z_i^\top\mid X_i=x)\tilde\gamma\\
    &=\mu''_Y-\mu''_{Z^\top}\tilde\gamma,
\end{align*}
we obtain
\begin{equation*}
    \mu''_{\tilde Y+}-\mu''_{\tilde Y-}=\mu''_{Y+}-\mu''_{Y-}-\left(\mu''_{Z+}-\mu''_{Z-}\right)^\top\tilde\gamma.
\end{equation*}
Therefore it is enough to show that
\begin{equation*}
    \left(\mu''_{Z+}-\mu''_{Z-}\right)^\top\left(\check\beta_n-\tilde\gamma\right)=o(1).
\end{equation*}
Since the first factor is constant, we just have to evaluate the convergence of the second factor. Indeed,
\begin{align*}
    \norm{\check\beta_n-\tilde\gamma}_2&=\norm{\E\left(K_h(X_i)\Zit\tilde Z_i^{\top}\right)^{-1}\E\left(K_h(X_i)\Zit \left(Y_i-\tilde Z_i^\top\tilde\gamma\right)\right)}_2\\
    &\leq O(1)\norm{\E\left(K_h(X_i)\Zit \left(Y_i-\tilde Z_i^\top\tilde\gamma\right)\right)}_2
\end{align*}
whereby Assumption \ref{assumption4} is used. Now set
\begin{equation*}
    f(x)=\E\left(\Zit \left(Y_i-\tilde Z_i^\top\tilde\gamma\right)\;\,\vrule\; X_i=x\right),
\end{equation*}
then
\begin{equation*}
    \E\left(K_h(X_i)\Zit \left(Y_i-\tilde Z_i^\top\tilde\gamma\right)\right)=\E\left(\frac{1}{h}K\left(\frac{X_i}{h}\right)f(X_i)\right).
\end{equation*}
Note that $f(x)$ does not depend on $h$ and that we have
\begin{equation}\label{blue3}
\begin{split}
    f_+=\;&\mu_{\tilde Z Y+}-\mu_{\tilde Z\tilde Z^\top \tilde\gamma+}\\
    =\;&\mu_{(Z-M_n^\top V)Y+}-\mu_{(Z-M_n^\top V)(Z-M_n^\top V)^\top+}\tilde\gamma\\
    =\;&\mu_{ZY+}-\mu_{Z+}\mu_{Y+}+\mu_{Z+}\mu_{Y+}-\mu_{M_n^\top VY+}-(\mu_{ZZ^\top+}-\mu_{Z+}\mu_{Z^\top +})\tilde\gamma\\
    &+(-\mu_{Z+}\mu_{Z^\top +}+\mu_{M_n^\top V Z^\top+}+\mu_{Z V^\top M_n+}-\mu_{M_n^\top VV^\top M_n+})\tilde\gamma\\
    =\;&\sigma_{ZY+}^2-\sigma_{Z+}^2\tilde\gamma\\
    &+\mu_{Z+}\mu_{Y+}-\mu_{M_n^\top VY+}\\
    &+(-\mu_{Z+}\mu_{Z^\top +}+\mu_{M_n^\top V Z^\top+}+\mu_{Z V^\top M_n+}-\mu_{M_n^\top VV^\top M_n+})\tilde\gamma\\
    =\;&\sigma_{ZY+}^2-\sigma_{Z+}^2\tilde\gamma.
\end{split}
\end{equation}
Analogously we obtain $f_-=\sigma_{ZY-}^2-\sigma_{Z-}^2\tilde\gamma$, leading to $f_++f_-=0$.

We perform a Taylor expansion similarly to (\ref{taylorexpansion}), but this time just up to the first degree, which means
\begin{align*}
\E\left(K_h(X_i)\Zit \left(Y_i-\tilde Z_i^\top\tilde\gamma\right)\right)&=K_-^{(0)}f_-f_X(0)+K_+^{(0)}f_+f_X(0)+O(h)\\
&=K_-^{(0)}f_X(0)\underbrace{(f_-+f_+)}_{=0}+O(h)=O(h).
\end{align*}
Hence there is some constant $C>0$ such that
\begin{equation*}
    \norm{\check\beta_n-\tilde\gamma}_2\leq Ch\xrightarrow{n\to\infty}0,
\end{equation*}
which in turn implies the statement of the Proposition.
\end{proof}
    
\end{subsection}

\begin{subsection}{Convergence to the Standard Normal Distribution}\label{convergence}

This chapter lays the foundation for proving the convergence to the standard normal distribution in the main theorem on asymptotic normality of the covariate-adjusted estimator. The result of the following proposition is used by applying it on an adjusted version of the covariates, namely $\tilde Z_i=Z_i-M_n^\top V_i$. By doing that, some additional properties concerning the value of $\mu_{\tilde Z}$ and of the one-sided limits of $\mu'_{\tilde Z}$ in zero emerge, which are necessary to fulfill the assumptions of the proposition contained in this chapter.

As we need a matrix notation for our population data in order to write some of the steps in the proof of this section's proposition more compactly and concisely, the following definition will be highly beneficial.

\begin{definition}
Let $K$ be a kernel function, then we set
\begin{equation*}
    \Kh=\mathrm{diag}\left(\frac{1}{h}K\left(\frac{X_1}{h}\right),...,\frac{1}{h}K\left(\frac{X_n}{h}\right)\right)
\end{equation*}
and
\begin{equation*}
    \Kh^{\frac{1}{2}}=\mathrm{diag}\left(\sqrt{\frac{1}{h}K\left(\frac{X_1}{h}\right)},...,\sqrt{\frac{1}{h}K\left(\frac{X_n}{h}\right)}\right).
\end{equation*}
Furthermore, we define
\begin{align*}
    \Ybf=\begin{pmatrix}
        Y_1\\
        \vdots\\
        Y_n
    \end{pmatrix}\in\mathbb{R}^n,\quad \Zbf=\begin{pmatrix}
        Z_1^\top\\
        \vdots\\
        Z_n^\top
    \end{pmatrix}\in\mathbb{R}^{n\times p}, \quad\rbf(h)=\begin{pmatrix}
        r_1(h)\\
        \vdots \\
        r_n(h)
    \end{pmatrix}\in\mathbb{R}^n
\end{align*}
and
\begin{equation*}
    \Vbf=\begin{pmatrix}
        1 & T_1 & X_1/h & T_1X_1/h\\
        \vdots & \vdots & \vdots & \vdots \\
        1 & T_n & X_n/h & T_nX_n/h
    \end{pmatrix}\in\mathbb{R}^{n\times 4}.
\end{equation*}
\end{definition}

\begin{proposition}\label{theoremA1}
Let the assumptions in (\ref{cond1}) of Lemma \ref{propositionA7} hold. Further, let $K$ be a kernel which is compactly supported on $[-1,1]$ with $K^{(4)},(K^2)^{(0)}<\infty$ and suppose that $\E(K_h(X_i)Z_iZ_i^\top)$ is invertible with $\norm{\E(K_h(X_i)Z_iZ_i^\top)^{-1}}_2=O(1)$. Also, let $f_X$ be three times continuously differentiable in a neighborhood around zero with $f_X(0)>0$. Moreover, assume that $\mu_Z$ is continuous and three times one-sided differentiable at zero (cf. Definition \ref{defofonesideddiff}), whereby $\mu_Z(0)=0$ and $\lim_{x\searrow 0} \mu'_Z(x)=\lim_{x\nearrow 0} \mu'_Z(x)=0$. Furthermore, suppose that there is $\delta>0$ such that
\begin{equation*}
\sup_{n\in\mathbb{N}}\sup_{x\in[-h,h]} \E\left(|r_i(h)|^{2+\delta}\mid X_i=x\right) < \infty
\end{equation*}
as well as the existence of finite numbers $\sigma_l,\sigma_r>0$ such that
\begin{equation}\label{assump}
\begin{split}
&\lim_{n\to\infty}\sup_{\lambda\in [0,1]} |\E\left(r_i(h)^2\mid X_i=\lambda h\right)-\sigma_r^2|=0,\\
&\lim_{n\to\infty}\sup_{\lambda\in [0,1]} |\E\left(r_i(h)^2\mid X_i=-\lambda h\right)-\sigma_l^2|=0.
\end{split}
\end{equation}
Set $w=\left(\left[(f_X(0)\kappa(K))^{-1}\right]_{2\cdot}\right)^\top$ and define
\begin{equation*}
    \mathcal{S}_n^2=\frac{1}{h}\E\left(K\left(\frac{X_i}{h}\right)^2(w^\top V_i)^2r_i(h)^2\right).
\end{equation*}
Then,
\begin{equation*}
    \sqrt{\frac{nh}{\mathcal{S}_n^2}}\left(\hat{\tau}_h-\theta_{0}^{(2)}\right)\xrightarrow[]{d}\mathcal{N}(0,1).
\end{equation*}
\end{proposition}
\begin{proof}
Denote by
\begin{equation*}
    M=I-\Khfrac \Zbf(\Ztbf \Kh \Zbf)^{-1}\Ztbf \Khfrac
\end{equation*}
the projection matrix. First note that
\begin{equation}\label{equalszero}
    M \Khfrac \Zbf=0
\end{equation}
by using the definition of $M$. Now, define $\bar X:=\Khfrac\begin{pmatrix}\Vbf&\Zbf\end{pmatrix}$ and $\bar Y:=\Khfrac\Ybf$. The linear least squares problem
\begin{equation*}
    (\thetanhat,\gammanhat)=\argmin_{(\theta, \gamma)\in\mathbb{R}^{p+4}} \sum_{i=1}^n K_h(X_i)\left(Y_i-V_i^\top\theta-Z_i^\top\gamma\right)^2
\end{equation*}
can be written as
\begin{equation*}
    \hat\beta=\argmin_{\beta\in\mathbb{R}^{p+4}}\, \norm{\bar Y-\bar X\beta}_2^2,
\end{equation*}
which has the solution
\begin{equation*}
    \hat\beta=(\thetanhat,\gammanhat)=(\bar X^\top \bar X)^{-1}\bar X^\top \bar Y.
\end{equation*}
We calculate this expression for $\hat\beta$. In fact, as we have
\begin{equation*}
    \bar X^\top \bar X=\begin{pmatrix}
        \Vbf^\top \Kh \Vbf & \Vbf^\top \Kh \Zbf\\
        \Ztbf\Kh\Vbf & \Ztbf\Kh\Zbf
    \end{pmatrix},
\end{equation*}
we obtain by applying the well-known formula for the inverse of block matrices that
\begin{equation}\label{eq1}
    (\bar X^\top \bar X)^{-1} = \begin{pmatrix}
        R^{-1} & -R^{-1}\Vbf^\top \Kh\Zbf(\Ztbf\Kh\Zbf)^{-1}\\
        *&*
    \end{pmatrix}
\end{equation}
whereby
\begin{align*}
    R:=\;&\Vbf^\top\Kh\Vbf-\Vbf^\top\Kh\Zbf(\Ztbf\Kh\Zbf)^{-1}\Ztbf\Kh\Vbf\\
    =\;&\Vbf^\top\Khfrac(I-\Khfrac \Zbf(\Ztbf \Kh \Zbf)^{-1}\Ztbf \Khfrac)\Khfrac\Vbf\\
    =\;&\Vbf^\top\Khfrac M\Khfrac\Vbf.
\end{align*}
Furthermore,
\begin{equation}\label{eq2}
    \bar X^\top\bar Y=\begin{pmatrix}
        \Vbf^\top \Kh\Ybf\\
        \Zbf^\top \Kh\Ybf
    \end{pmatrix}.
\end{equation}
Taking (\ref{eq1}) and (\ref{eq2}) together delivers
\begin{align*}
    \thetanhat&=(\Vbf^\top\Khfrac M\Khfrac\Vbf)^{-1}(\Vbf^\top K_h\Ybf - \Vbf^\top \Kh\Zbf(\Ztbf\Kh\Zbf)^{-1}\Zbf^\top K_h\Ybf)\\
    &=(\Vbf^\top\Khfrac M\Khfrac\Vbf)^{-1}(\Vbf^\top\Khfrac(I-\Khfrac \Zbf(\Ztbf \Kh \Zbf)^{-1}\Ztbf \Khfrac)\Khfrac\Ybf)\\
    &=(\Vbf^\top\Khfrac M\Khfrac\Vbf)^{-1}\Vbf^\top\Khfrac M \Khfrac\Ybf.
\end{align*}
Considering that
\begin{equation*}
    \Ybf=\rbf(h)+\Vbf\theta_0(h) + \Zbf\gamma_0(h)
\end{equation*}
finally delivers the representation
\begin{equation}\label{equa0}
\begin{split}
    \thetanhat=\;\,&(\Vbf^\top\Khfrac M\Khfrac\Vbf)^{-1}\Vbf^\top\Khfrac M \Khfrac(\rbf(h)+\Vbf\theta_0(h) + \Zbf\gamma_0(h))\\
    =\;\,&(\Vbf^\top\Khfrac M\Khfrac\Vbf)^{-1}\Vbf^\top\Khfrac M \Khfrac(\Zbf\gamma_0(h)+\rbf(h))\\
    &+\underbrace{(\Vbf^\top\Khfrac M\Khfrac\Vbf)^{-1}\Vbf^\top\Khfrac M \Khfrac \Vbf}_{=I}\theta_0(h)\\
    =\;\,&\theta_0(h)+(\Vbf^\top\Khfrac M\Khfrac\Vbf)^{-1}\Vbf^\top\Khfrac M \Khfrac(\Zbf\gamma_0(h)+\rbf(h))\\
    \overset{(\ref{equalszero})}{=}&\theta_0(h)+(\Vbf^\top\Khfrac M\Khfrac\Vbf)^{-1}\Vbf^\top\Khfrac M \Khfrac\rbf(h).
\end{split}
\end{equation}
By applying Lemma \ref{propositionA7}, we obtain
\begin{equation}\label{equa1}
    \Vbf^\top\Khfrac M \Khfrac\rbf(h)=\frac{1}{n}\Vbf^\top \Kh\rbf(h) + o_P\left(\frac{1}{\sqrt{nh}}\right)
\end{equation}
and
\begin{equation}\label{equa2}
    (\Vbf^\top\Khfrac M\Khfrac\Vbf)^{-1}=(f_X(0)\kappa(K)+o_P(1))^{-1}=(f_X(0)\kappa(K))^{-1}+o_P(1).
\end{equation}
Hence, inserting (\ref{equa1}) and (\ref{equa2}) into (\ref{equa0}), gives
\begin{equation}\label{toExpand}
    \thetanhat = \theta_0 + ((f_X(0)\kappa(K))^{-1}+o_P(1))\left(\frac{1}{n}\Vbf^\top \Kh\rbf(h) + o_P\left(\frac{1}{\sqrt{nh}}\right)\right).
\end{equation}
We know by Lemma \ref{lemmaA8}, whose assumption (\ref{rsquared}) is satisfied by (\ref{assump}), that
\begin{equation*}
   o_P(1)\frac{1}{n}\Vbf^\top\Kh\rbf(h)=o_P(1)O_P\left(\frac{1}{\sqrt{nh}}\right)=o_P\left(\frac{1}{\sqrt{nh}}\right).
\end{equation*}
Therefore, (\ref{toExpand}) expands to
\begin{equation*}
     \thetanhat=\theta_0 + (f_X(0)\kappa(K))^{-1}\frac{1}{n}\Vbf^\top \Kh\rbf(h) + (f_X(0)\kappa(K))^{-1}o_P\left(\frac{1}{\sqrt{nh}}\right) + o_P\left(\frac{1}{\sqrt{nh}}\right),
\end{equation*}
which, overall, leads to
\begin{equation}\label{secondrow}
    \sqrt{nh}\left( \thetanhat-\theta_0 \right) = (f_X(0)\kappa(K))^{-1}\sqrt{nh}\frac{1}{n}\Vbf^\top \Kh\rbf(h) + o_P(1).
\end{equation}
Since we are interested in the estimator of the average treatment effect, we need to study the second row of that expression. Please recall the definition of $w$ given in the assertion of the proposition. Then, the second row of (\ref{secondrow}) can be written as
\begin{equation*}
    \sqrt{nh}\left( \hat\tau_h - \theta_0^{(2)} \right)=\sqrt{nh}\frac{1}{n}w^\top\Vbf^\top\Kh\rbf(h) + o_P(1)=\frac{1}{\sqrt{nh}}\sum_{i=1}^n K\left( \frac{X_i}{h} \right)w^\top V_i r_i(h) + o_P(1).
\end{equation*}
Next, we want to use Lyapunov's Central Limit Theorem to argue that the result follows. Define $\nu(X_i/h):=w^\top V_i$. Then, the independent sequence of random variables is considered to be
\begin{equation*}
    \left(\frac{1}{\sqrt{nh}}K\left(\frac{X_i}{h} \right)\nu\left(\frac{X_i}{h}\right)r_i(h)\right)_{i=1,...,n}
\end{equation*}
whereby each of this random variables has expectation zero due to the definition of the residual. We calculate the variance $\mathcal{S}_n^2$ of the sum of the variables now:
\begin{align*}
    \mathcal{S}_n^2&=\mathrm{Var}\left(\frac{1}{\sqrt{nh}}\sum_{i=1}^n K\left(\frac{X_i}{h} \right)\nu\left(\frac{X_i}{h}\right)r_i(h)\right)\\
    &=\frac{1}{nh}\sum_{i=1}^n \E\left(K\left(\frac{X_i}{h} \right)^2\nu\left(\frac{X_i}{h}\right)^2r_i(h)^2\right)\\
    &=\frac{1}{h}\E\left(K\left(\frac{X_i}{h} \right)^2\nu\left(\frac{X_i}{h}\right)^2r_i(h)^2\right).
\end{align*}
Now we need to state that the required condition
\begin{equation}\label{reqCondCLT}
\lim_{n\to\infty} \frac{1}{\mathcal{S}_n^{2+\delta}}\sum_{i=1}^n \frac{1}{\left(\sqrt{nh}\right)^{2+\delta}}\E\left(K\left(\frac{X_i}{h} \right)^{2+\delta}\left|\nu\left(\frac{X_i}{h}\right)\right|^{2+\delta}\left|r_i(h)\right|^{2+\delta}\right)=0
\end{equation}
for Lyapunov's CLT is true. Indeed, there are constants $C_1, C_2>0$ and $\delta>0$ such that
\begin{align}
    &\frac{1}{h}\E\left(K\left(\frac{X_i}{h} \right)^2\nu\left(\frac{X_i}{h}\right)^2r_i(h)^2\right)\xrightarrow[n\to\infty]{}C_1,\label{C1}\\
    &\frac{1}{h}\E\left(K\left(\frac{X_i}{h} \right)^{2+\delta}\left|\nu\left(\frac{X_i}{h}\right)\right|^{2+\delta}\left|r_i(h)\right|^{2+\delta}\right)\leq C_2.\label{C2}
\end{align}
The convergence in (\ref{C1}) can be seen by
\begin{align*}
    &\lim_{n\to\infty} \frac{1}{h}\E\left(K\left(\frac{X_i}{h} \right)^2\nu\left(\frac{X_i}{h}\right)^2r_i(h)^2\right)\\
    =\;&\lim_{n\to\infty}\frac{1}{h}\E\left(K\left(\frac{X_i}{h} \right)^2\nu\left(\frac{X_i}{h}\right)^2\E(r_i(h)^2\mid X_i)\right)\\
    =\;&\lim_{n\to\infty}\int_{-\infty}^\infty K(y)^2\nu(y)^2\E(r_i(h)^2\mid X_i=yh)f_X(yh)\,dy\\
    =\;&\int_{-1}^1 K(y)^2\nu(y)^2\lim_{n\to\infty}\E(r_i(h)^2\mid X_i=yh)f_X(yh)\,dy\\
    =\;&\int_{-1}^0 K(y)^2\nu(y)^2\lim_{n\to\infty}\E(r_i(h)^2\mid X_i=yh)f_X(yh)\,dy\\
    &+\int_{0}^1 K(y)^2\nu(y)^2\lim_{n\to\infty}\E(r_i(h)^2\mid X_i=yh)f_X(yh)\,dy
\end{align*}
whereby we used the Theorem of Dominated Convergence to move the limit inside the integral. For both $1\geq y>0$ and $-1\leq y<0$, $\lim_{n\to\infty}\E(r_i(h)^2\mid X_i=yh)$ exists directly due to the assumptions stated in (\ref{assump}). Thus, the above expression converges to some $C_1>0$.\\
The upper bound in (\ref{C2}) is an immediate consequence of the assumption
\begin{equation*}
\sup_{n\in\mathbb{N}}\sup_{x\in[-h,h]}\;\E(|r_i(h)|^{2+\delta}\mid X_i=x)<\infty
\end{equation*}
since
\begin{align*}
    &\frac{1}{h}\E\left(K\left(\frac{X_i}{h} \right)^{2+\delta}\left|\nu\left(\frac{X_i}{h}\right)\right|^{2+\delta}\left|r_i(h)\right|^{2+\delta}\right)\\
    =\;&\int_{-1}^1 K(y)^{2+\delta}|\nu(y)|^{2+\delta}\E(|r_i(h)|^{2+\delta}\mid X_i=yh)f_X(yh)\,dy.
\end{align*}
Combining (\ref{C1}) and (\ref{C2}) gives
\begin{align*}
    &\frac{1}{\mathcal{S}_n^{2+\delta}}\sum_{i=1}^n \frac{1}{\left(\sqrt{nh}\right)^{2+\delta}}\E\left(K\left(\frac{X_i}{h} \right)^{2+\delta}\left|\nu\left(\frac{X_i}{h}\right)\right|^{2+\delta}\left|r_i(h)\right|^{2+\delta}\right)\\
    =\;&\sum_{i=1}^n \frac{n^{-\frac{2+\delta}{2}}h^{-\frac{2+\delta}{2}}\E\left(K\left(\frac{X_i}{h} \right)^{2+\delta}\left|\nu\left(\frac{X_i}{h}\right)\right|^{2+\delta}\left|r_i(h)\right|^{2+\delta}\right)}{\mathcal{S}_n^{2+\delta}}\\
    =\;& \frac{1}{n}\sum_{i=1}^n \frac{n^{-\frac{\delta}{2}}h^{-\frac{\delta}{2}}\frac{1}{h}\E\left(K\left(\frac{X_i}{h} \right)^{2+\delta}\left|\nu\left(\frac{X_i}{h}\right)\right|^{2+\delta}\left|r_i(h)\right|^{2+\delta}\right)}{\mathcal{S}_n^{2+\delta}}\\
    \leq\;& \frac{C_2(nh)^{-\frac{\delta}{2}}}{\left(\frac{C_1}{2}\right)^{\frac{2+\delta}{2}}}\xrightarrow[n\to\infty]{} 0\quad\text{(since $nh\to\infty$)}
\end{align*}
which is condition (\ref{reqCondCLT}) we wanted to show. Hence, as Lyapunov's CLT can be applied, we obtain
\begin{equation*}
    \frac{1}{\sqrt{nh\mathcal{S}_n^2}}\sum_{i=1}^n K\left(\frac{X_i}{h}\right)\nu\left(\frac{X_i}{h}\right)r_i(h)\xrightarrow[]{d}\mathcal{N}(0,1).
\end{equation*}
This shows the assertion.
\end{proof}

\end{subsection}

\begin{subsection}{Conversion of the Variance Representation}\label{varianceConversion}
In this section, we want to transform the representation of the variance provided in Section \ref{convergence} in order to obtain the one stated in the main theorem's assertion.

\begin{proposition}[Computation for the variance]\label{lemmaVar}
Let the assumptions of Section \ref{secAssump} hold. Then
\begin{equation*}
    \frac{1}{h}\E\left(K\left(\frac{X_i}{h}\right)^2(w^\top V_i)^2\check r_i(h)^2\right)+o(1)=\Sn^2
\end{equation*}
whereby $w=\left(\left[(f_X(0)\kappa(K))^{-1}\right]_{2\cdot}\right)^\top$.
\end{proposition}
\begin{proof}
We conclude in two steps. First, we show that
\begin{equation}\label{stepOne}
    \frac{1}{h}\E\left(K\left(\frac{X_i}{h}\right)^2(w^\top V_i)^2\check r_i(h)^2\right)\xrightarrow{h\to 0}\frac{\mathcal{C_S}}{f_X(0)}\left(\sigma_l^2+\sigma_r^2\right),
\end{equation}
whereby $\sigma_l^2$ and $\sigma_r^2$ are the finite numbers provided by Assumption \ref{assumption7}. Indeed,
\begin{align*}
    &\frac{1}{h}\E\left(K\left(\frac{X_i}{h}\right)^2(w^\top V_i)^2\check r_i(h)^2\right)\\
    =\;&\frac{1}{h}\E\left(K\left(\frac{X_i}{h}\right)^2(w^\top V_i)^2\E\left(\check r_i(h)^2\mid X_i\right)\right)\\
    =\;&\int_{-\infty}^\infty \frac{1}{h}K\left(\frac{x}{h}\right)^2\left(w^\top\begin{pmatrix}
        1 & \ind & \frac{x}{h} & \frac{x}{h}\ind
    \end{pmatrix}^\top\right)^2\mu_{\check r_i(h)^2}(x)f_X(x)\;dx\\
    =\;&\int_{-\infty}^\infty K(y)^2\frac{1}{f_X(0)^2}\left([\kappa(K)^{-1}]_2\begin{pmatrix}
        1 & \mathds{1}(yh\geq 0) & y & y\mathds{1}(yh\geq 0)
    \end{pmatrix}^\top\right)^2\\
    &\quad\quad\;\mu_{\check r_i(h)^2}(yh)f_X(yh)\;dy\\
    =\;&\int_{-1}^0 K(y)^2\frac{1}{f_X(0)^2}\left([\kappa(K)^{-1}]_2\begin{pmatrix}
        1 & 0 & y & 0
    \end{pmatrix}^\top\right)^2\mu_{\check r_i(h)^2}(yh)f_X(yh)\;dy\\
    &+ \int_{0}^1 K(y)^2\frac{1}{f_X(0)^2}\left([\kappa(K)^{-1}]_2\begin{pmatrix}
        1 & 1 & y & y
    \end{pmatrix}^\top\right)^2\mu_{\check r_i(h)^2}(yh)f_X(yh)\;dy\\
    \xrightarrow{h\to 0}&\int_{-1}^0 K(y)^2\frac{1}{f_X(0)}\left([\kappa(K)^{-1}]_2\begin{pmatrix}
        1 & 0 & y & 0
    \end{pmatrix}^\top\right)^2\sigma_l^2\;dy\\
    &+ \int_{0}^1 K(y)^2\frac{1}{f_X(0)}\left([\kappa(K)^{-1}]_2\begin{pmatrix}
        1 & 1 & y & y
    \end{pmatrix}^\top\right)^2\sigma_r^2\;dy
\end{align*}
whereby the limit is moved inside the integral by the Theorem of Dominated Convergence. Also, we used that
\begin{equation*}
    \mu_{\check r_i(h)^2}(-yh)\xrightarrow[n\to\infty]{}\sigma_l^2\quad\text{ and }\quad\mu_{\check r_i(h)^2}(yh)\xrightarrow[n\to\infty]{}\sigma_r^2
\end{equation*}
for all $y\in [0,1]$ since $\check r_i(h)=r_i(h)$ and by the choice of $\sigma_l$ and $\sigma_r$ in Assumption \ref{assumption7}.

Next, we use the representation of $\kappa(K)^{-1}$ as stated in Lemma \ref{lemmaA14} and obtain
\begin{align*}
    \int_{-\infty}^0 K(y)^2\frac{1}{f_X(0)}\left([\kappa(K)^{-1}]_2\begin{pmatrix}
        1 & 0 & y & 0
    \end{pmatrix}^\top\right)^2\sigma_l^2\;dy=\frac{\mathcal{C_S}}{f_X(0)}\sigma_l^2
\end{align*}
and
\begin{align*}
\int_{0}^\infty K(y)^2\frac{1}{f_X(0)}\left([\kappa(K)^{-1}]_2\begin{pmatrix}
        1 & 1 & y & y
    \end{pmatrix}^\top\right)^2\sigma_r^2\;dy=\frac{\mathcal{C_S}}{f_X(0)}\sigma_r^2.
\end{align*}
Therefore, the convergence in (\ref{stepOne}) is proved. We proceed with the second step and want to show that
\begin{equation}\label{stepTwo}
    \sigma^2_{\tilde Y+}=\sigma_r^2\quad\text{and}\quad\sigma^2_{\tilde Y-}=\sigma_l^2.
\end{equation}
We do this by first using the representation of $\thetazerocheck(h)$ as stated in (\ref{repofthetazerocheck}). We have
\begin{equation}\label{convEq}
\begin{split}
    \thetazerocheck(h)=\;&\Big(I-\E\left(K_h(X_i)V_iV_i^\top\right)^{-1}\E\left(K_h(X_i)V_i\tilde Z_i^\top\right)\E\left(K_h(X_i)\tilde Z_i\tilde Z_i^\top\right)^{-1}\\&\E\left(K_h(X_i)\tilde Z_iV_i^\top\right)\Big)^{-1}\left(\kappa(K)^{-1}\E\left(K_h(X_i)V_iV_i^\top\right)\right)^{-1}\\
    &\kappa(K)^{-1}\bigg(\E(K_h(X_i)V_iY_i)-\E\left(K_h(X_i)V_i\tilde Z_i^\top\right)\E\left(K_h(X_i)\tilde Z_i\tilde Z_i^\top\right)^{-1}\\
    &\E(K_h(X_i)\tilde Z_iY_i)\bigg)\\
    =\;&\left(I-\left(\frac{1}{f_X(0)}\kappa(K)^{-1}+o(1)\right)\cdot o(1)\right)^{-1}\cdot \left(\frac{1}{f_X(0)}I+o(1)\right)\\
    &\cdot \left(\begin{pmatrix}
        f_X(0)\mu_{Y-} \\
        f_X(0)\tauy \\
        h[\mu_Yf_X]'_- \\
        h\left([\mu_Yf_X]'_+-[\mu_Yf_X]'_-\right)
    \end{pmatrix}+o(1)\right)=\begin{pmatrix}
        \mu_{Y-}\\
        \tauy \\
        0 \\
        0
    \end{pmatrix}+o(1)
\end{split}
\end{equation}
whereby we used Lemma \ref{lemmaA12}, Lemma \ref{lemmaA15} and Lemma \ref{lemmaA11} together with the fact that $\mu_{\tilde Z}(0)=\mu'_{\tilde Z}(0)=0$ to obtain the above convergence.

Next, we want to state that
\begin{equation}\label{repofgammazero}
    \gamma_0(h)=\tilde\gamma+O(h).
\end{equation}
To show this we use the representation given in (\ref{repofthetazerocheck}) again, but this time we interchange the roles of $V_i$ and $Z_i$ in order to represent $\gamma_0(h)$ instead of $\theta_0(h)$. Indeed, we obtain
\begin{equation}\label{gammazero}
\begin{split}
    &\gamma_0(h)\\
    =&\left(\E(K_h(X_i)Z_iZ_i^\top)-\E(K_h(X_i)Z_iV_i^\top)\E(K_h(X_i)V_iV_i^\top)^{-1}\E(K_h(X_i)V_iZ_i^\top)\right)^{-1}\\
    &\left(\E(K_h(X_i)Z_iY_i)-\E(K_h(X_i)Z_iV_i^\top)\E(K_h(X_i)V_iV_i^\top)^{-1}\E(K_h(X_i)V_iY_i)\right)\\
    =\;&\Big(\E(K_h(X_i)Z_iZ_i^\top)-\E(K_h(X_i)Z_iV_i^\top)\left(\kappa(K)^{-1}\E(K_h(X_i)V_iV_i^\top)\right)^{-1}\\
    &\kappa(K)^{-1}\E(K_h(X_i)V_iZ_i^\top)\Big)^{-1}\\
    &\Big(\E(K_h(X_i)Z_iY_i)-\E(K_h(X_i)Z_iV_i^\top)\left(\kappa(K)^{-1}\E(K_h(X_i)V_iV_i^\top)\right)^{-1}\\
    &\kappa(K)^{-1}\E(K_h(X_i)V_iY_i)\Big).
\end{split}
\end{equation}
We begin by examining the first factor. We have
\begin{align*}
    \E(K_h(X_i)Z_iZ_i^\top)&=\E(K_h(X_i)\E(Z_iZ_i^\top\mid X_i))\\
    &=\Kmn\mu_{ZZ^\top-}f_X(0)+\Kpn\mu_{ZZ^\top+}f_X(0)+O(h)\\
    &=\frac{f_X(0)}{2}\left(\mu_{ZZ^\top+}+\mu_{ZZ^\top-}\right)+O(h)
\end{align*}
by performing Taylor expansion as in (\ref{taylorexpansion}),
\begin{align*}
    \E(K_h(X_i)Z_iV_i^\top)&=\E\left(K_h(X_i)Z_i\begin{pmatrix}
        1 & T_i & \frac{X_i}{h} & \frac{T_iX_i}{h}
    \end{pmatrix}\right)\\
    &=\begin{pmatrix}
        \E(K_h(X_i)\E(Z_i\mid X_i)) & * & * & *
    \end{pmatrix}\\
    &=\begin{pmatrix}
        \frac{f_X(0)}{2}\left(\mu_{Z+}+\mu_{Z-}\right)+O(h) & * & * & *
    \end{pmatrix}
\end{align*}
also by Taylor expansion according to (\ref{taylorexpansion}),
\begin{align*}
    \kappa(K)^{-1}\E(K_h(X_i)V_iV_i^\top)=f_X(0)I+O(h)
\end{align*}
by Lemma \ref{lemmaA15}, and finally
\begin{align*}
    \kappa(K)^{-1}\E(K_h(X_i)V_iZ_i^\top)=\begin{pmatrix}
        f_X(0)\mu^\top_{Z-} \\ 0 \\ 0 \\ 0
    \end{pmatrix}+O(h)
\end{align*}
by applying Lemma \ref{lemmaA12} component-wise and the fact that $\mu_{Z+}=\mu_{Z-}$ due to the continuity of $\mu_Z$. Taking all the above statements together gives
\begin{align*}
    &\E(K_h(X_i)Z_iZ_i^\top)-\E(K_h(X_i)Z_iV_i^\top)\left(\kappa(K)^{-1}\E(K_h(X_i)V_iV_i^\top)\right)^{-1}\\
    &\kappa(K)^{-1}\E(K_h(X_i)V_iZ_i^\top)\\
    =\;&\frac{f_X(0)}{2}\left(\mu_{ZZ^\top+}+\mu_{ZZ^\top-}\right)-\begin{pmatrix}
        \frac{f_X(0)}{2}\left(\mu_{Z+}+\mu_{Z-}\right) & * & * & *
    \end{pmatrix}\frac{1}{f_X(0)}\begin{pmatrix}
        f_X(0)\mu^\top_{Z-} \\ 0 \\ 0 \\ 0
    \end{pmatrix}\\
    &+O(h)\\
    =\;&\frac{f_X(0)}{2}\left(\mu_{ZZ^\top+}+\mu_{ZZ^\top-}-\mu_{Z+}\mu^\top_{Z+}-\mu_{Z-}\mu^\top_{Z-}\right)+O(h)\\
    =\;&\frac{f_X(0)}{2}\left(\sigma^2_{Z+}+\sigma^2_{Z-}\right)+O(h)
\end{align*}
whereby we used again that $\mu^\top_{Z+}=\mu^\top_{Z-}$. In an analogous way, we obtain for the second factor
\begin{align*}
    &\E(K_h(X_i)Z_iY_i)-\E(K_h(X_i)Z_iV_i^\top)\left(\kappa(K)^{-1}\E(K_h(X_i)V_iV_i^\top)\right)^{-1}\kappa(K)^{-1}\E(K_h(X_i)V_iY_i)\\
    =\;& \frac{f_X(0)}{2}\left(\sigma^2_{ZY+}+\sigma^2_{ZY-}\right)+O(h).
\end{align*}
Referring back to equation (\ref{gammazero}) and inserting the respective terms, we overall obtain
\begin{equation}\label{convGammazero}
\begin{split}
    \gamma_0(h)&=\left(\frac{f_X(0)}{2}\left(\sigma^2_{Z+}+\sigma^2_{Z-}\right)+O(h)\right)^{-1}\left(\frac{f_X(0)}{2}\left(\sigma^2_{ZY+}+\sigma^2_{ZY-}\right)+O(h)\right)\\
    &=\left(\sigma^2_{Z+}+\sigma^2_{Z-}\right)^{-1}\left(\sigma^2_{ZY+}+\sigma^2_{ZY-}\right)+O(h)\\
    &=\tilde\gamma+O(h).
\end{split}
\end{equation}
We can use this equality in order to get
\begingroup\allowdisplaybreaks
\begin{align}\label{convThetazero}
    &\quad\;\theta_0(h)-\begin{pmatrix}
        \mu_{\tilde Y-} \\
        \mu_{\tilde Y+}-\mu_{\tilde Y-} \\
        0 \\
        0
    \end{pmatrix}\nonumber\\
    &=\thetazerocheck(h)-M_n\gamma_0(h)-\begin{pmatrix}
        \mu_{Y-}-\mu^\top_{Z-}\tilde\gamma \\
        \tau_Y \\
        0 \\
        0
    \end{pmatrix}\nonumber\\
    &=\thetazerocheck(h)-\begin{pmatrix}
        \mu_Z(0)^\top \\
        0 \\
        h\mu'^\top_{Z-} \\
        h(\mu'^\top_{Z+}-\mu'^\top_{Z-})
    \end{pmatrix}(\tilde\gamma+O(h))-\begin{pmatrix}
        \mu_{Y-}-\mu^\top_{Z-}\tilde\gamma \\
        \tau_Y \\
        0 \\
        0
    \end{pmatrix}\nonumber\\
    &=\thetazerocheck(h)-\begin{pmatrix}
        \mu_Z(0)^\top(\tilde\gamma+O(h)) \\
        0 \\
        h\mu'^\top_{Z-}(\tilde\gamma+O(h)) \\
        h(\mu'^\top_{Z+}-\mu'^\top_{Z-})(\tilde\gamma+O(h))
    \end{pmatrix}-\begin{pmatrix}
        \mu_{Y-}-\mu^\top_{Z-}\tilde\gamma \\
        \tau_Y \\
        0 \\
        0
    \end{pmatrix}\nonumber\\
    &\xrightarrow[h\to 0]{(\ref{convEq})}\begin{pmatrix}
        \mu_{Y-} \\
        \tau_Y \\
        0\\
        0
    \end{pmatrix}-\begin{pmatrix}
        \mu^\top_{Z-}\tilde\gamma \\
        0 \\
        0 \\
        0
    \end{pmatrix}-\begin{pmatrix}
        \mu_{Y-}-\mu^\top_{Z-}\tilde\gamma \\
        \tau_Y \\
        0 \\
        0
    \end{pmatrix}=0.
\end{align}
\endgroup
Next, we have
\begin{equation}\label{riyi}
\begin{split}
    &\lim_{x\searrow 0} \E(r_i^2(h)\mid X_i=x)-\lim_{x\searrow 0} \mathrm{Var}(\tilde Y_i\mid X_i=x)\\
    =\;&\lim_{x\searrow 0} \Big(\E\left((Y_i-V_i^\top\theta_0(h)-Z_i^\top\gamma_0(h))^2\mid X_i=x\right)- \E\left((Y_i-Z_i^\top\tilde\gamma)^2\mid X_i=x\right)\\
    &+\E\left(Y_i-Z_i^\top\tilde\gamma\mid X_i=x\right)^2\Big)\\
    =\;&\lim_{x\searrow 0}\Big(\E\left(Y_i^2\mid X=x\right)-2\E\left(Y_iV_i^\top\theta_0(h)\mid X_i=x\right)-2\E\left(Y_iZ_i^\top\gamma_0(h)\mid X_i=x\right)\\
    &+\E\left((V_i^\top\theta_0(h))^2\mid X_i=x\right)
    +2\E\left(V_i^\top\theta_0(h)Z_i^\top\gamma_0(h)\mid X_i=x\right)\\
    &+2\E\left((Z_i^\top\gamma_0(h))^2\mid X_i=x\right)
    -\E\left(Y_i^2\mid X_i=x\right)+\E\left(Y_i\mid X_i=x\right)^2\\
    &-\E\left((Z_i^\top\tilde\gamma)^2\mid X_i=x\right)+\E\left(Z_i^\top\tilde\gamma\mid X_i=x\right)^2
    +2\E\left(Y_iZ_i^\top\tilde\gamma\mid X_i=x\right)\\
    &-2\E\left(Y_i\mid X_i=x\right)\E\left(Z_i^\top\tilde\gamma\mid X_i=x\right)
    \Big)\\
    =\;&\theta_0(h)^\top\mu_{VV^\top+}\theta_0(h)-\left(\mu_{Y+}-\mu^\top_{Z+}\tilde\gamma\right)^2\\
    &+\gamma_0(h)^\top\mu_{ZZ^\top+}\gamma_0(h)-\gamma^\top_n\mu_{ZZ^\top+}\tilde\gamma\\
    &-2\mu_{YV^\top+}\theta_0(h)+2\mu^2_{Y+}-2\mu^\top_{Z+}\mu_{Y+}\tilde\gamma\\
    &-2\mu_{YZ^\top+}\gamma_0(h)+2\mu^\top_{YZ^\top+}\tilde\gamma\\
    &+2\theta_0(h)^\top\mu_{VZ^\top+}\gamma_0(h)+2\tilde\gamma^\top\mu_{Z+}\mu^\top_{Z+}\tilde\gamma-2\mu_{Y+}\mu^\top_{Z+}\tilde\gamma.
\end{split}
\end{equation}
Each line of the above expression, considered separately, converges to zero for $h\to 0$ by using (\ref{convGammazero}) as well as (\ref{convThetazero}), leading to the whole expression converging to zero. Therefore, we can conclude from equation (\ref{riyi}) that
\begin{equation*}
    \lim_{x\searrow 0} \E(r_i^2(h)\mid X_i=x)\xrightarrow[n\to\infty]{}\lim_{x\searrow 0}\mathrm{Var}\left(\tilde Y_i\mid X_i=x\right).
\end{equation*}
Completely analogously, we also obtain
\begin{equation*}
    \lim_{x\nearrow 0} \E(r_i^2(h)\mid X_i=x)\xrightarrow[n\to\infty]{}\lim_{x\nearrow 0}\mathrm{Var}\left(\tilde Y_i\mid X_i=x\right).
\end{equation*}
By definition of $\sigma_l$ and $\sigma_r$, we can follow that
\begin{equation*}
    \sigma^2_{\tilde Y+}=\sigma_r^2\quad\text{ and }\quad\sigma^2_{\tilde Y-}=\sigma_l^2,
\end{equation*}
which is exactly statement (\ref{stepTwo}) that we wanted to show. Hence, the whole assertion is proved.
\end{proof}
\end{subsection}

\newpage
\begin{appendices}
\section{Computing the Representation of the Bias}

\begin{lemma}\label{matrixinv}
Let $K$ be a kernel with $K^{(2)}<\infty$. Then the matrix
\begin{equation*}
    \kappa(K)=\begin{pmatrix}
    \Kn&\Kpn&\Ke&\Kpe\\
    \Kpn&\Kpn&\Kpe&\Kpe\\
    \Ke&\Kpe&\Kz&\Kpz\\
    \Kpe&\Kpe&\Kpz&\Kpz
    \end{pmatrix}
\end{equation*}
is invertible.
\end{lemma}
\begin{proof}
We prove that the determinant of the matrix is unequal to zero. The fact that
\begin{align}\label{K_properties}
    \Kn=1, \Kpn=\Kmn=\frac{1}{2}, \Kpe=-\Kme, \Ke=0 \text{ and } \Kpz=\Kmz
\end{align}
is useful to calculate that
\begingroup
\allowdisplaybreaks
\begin{equation*}
    \det \kappa(K)=\left((\Kpe)^2-\frac{1}{2}\Kpz\right)^2.
\end{equation*}
\endgroup
This expression is not zero since
\begin{align*}
    \left(\Kpe\right)^2&=\left(\frac{1}{2}\int_{-\infty}^\infty K(u)|u|\,du\right)^2<\frac{1}{4}\int_{-\infty}^\infty K(u)u^2\,du=\frac{1}{2}\Kpz,
\end{align*}
whereby the inequality is true due to Jensen's inequality which can be applied because our imposed assumptions make $K$ a probability density function. In particular, since it is applied on a non-linear function (namely quadratic function), the implied inequality is strict. Overall, this proves the invertibility of $\kappa(K)$.
\end{proof}

Please recall that the probability density function of $X_i$ is denoted by $f_X$. For the subsequent considerations we need the following basic statement: Let $L$ be an integrable function and $f:\mathbb{R}\to\mathbb{R}$ be a twice one-sided differentiable function at zero. Also let $f_X$ be twice continuously differentiable. Then, we can perform a Taylor expansion of $f$ at zero respectively from the left and from the right by evaluating the function at zero by using the continuous extension. This leads to the Taylor expansion of $L(u)f(uh)f_X(uh)$ at $h=0$ from the left, namely
\begin{align*}
    &L(u)f(uh)f_X(uh)=L(u)f_-f_X(0)+hL(u)u(f\cdot f_X)'_-+\frac{h^2}{2}L(u)u^2(f\cdot f_X)''_-+o(h^2)
\end{align*}
for $h\to 0$. This can be done analogously from the right side at zero for $h\to 0$. As a consequence, we obtain under the additional assumption $L^{(2)}<\infty$ that
\begin{equation}\label{taylorexpansion}
\begin{split}
    \E\left(\frac{1}{h}L\left(\frac{X_i}{h}\right)f(X_i)\right)=\;&L_-^{(0)}f_-f_X(0)+L_+^{(0)}f_+f_X(0)\\
    &+h\left[L_-^{(1)}(f\cdot f_X)'_-+L_+^{(1)}(f\cdot f_X)'_+\right]\\
    &+\frac{h^2}{2}\left[L_-^{(2)}(f\cdot f_X)''_-+L_+^{(2)}(f\cdot f_X)''_+\right]+o(h^2)
\end{split}
\end{equation}
for $h\to 0$. Please note that we used $\int o(h^2)\,du=o(h^2)$ which can be verified by taking a representation of the remainder term of the Taylor expansion (e.g. the Peano form) and using the Theorem of Dominated Convergence to move the limit inside the integral.

\begin{lemma}\label{lemmaA12}
Let $f_X$ be three times continuously differentiable in a neighborhood around zero and $A$ a random variable such that the function $\mu_A=\E(A\mid X_i=x)$ is well-defined and three times one-sided differentiable at $0$. Moreover, let K be a kernel with $K^{(4)}<\infty$. As we set
\begin{equation*}
    B(K,A):=\frac{1}{2}\kappa(K)^{-1}\left[ \begin{pmatrix}
    \Kpz\\\Kpz\\\Kpd\\\Kpd
    \end{pmatrix}[\mu_Af_X]''_+ + \begin{pmatrix}
    \Kmz\\0\\\Kmd\\0
    \end{pmatrix}[\mu_Af_X]''_- \right],
\end{equation*}
it holds that
\begin{equation*}
    \kappa(K)^{-1}\E(K_h(X_i)V_iA)=\begin{pmatrix}
    f_X(0)\mu_{A-}\\f_X(0)\tau_A\\h[\mu_Af_X]'_-\\h([\mu_Af_X]'_+-[\mu_Af_X]'_-)
    \end{pmatrix}+h^2B(K,A)+O(h^3)
\end{equation*}
for $h\to 0$.
\end{lemma}
\begin{proof}
We have
\begin{equation}\label{toinsert}
    \E(K_h(X_i)V_iA)=\E(\E(K_h(X_i)V_iA\mid X_i))=\begin{pmatrix}
        \E(K_h(X_i)\E(A\mid X_i))\\
        \E(K_h(X_i)T_i\E(A\mid X_i))\\
        \E(K_h(X_i)\frac{X_i}{h}\E(A\mid X_i))\\
        \E(K_h(X_i)\frac{T_iX_i}{h}\E(A\mid X_i))
    \end{pmatrix}.
\end{equation}
We perform a Taylor expansion for each of the components as in (\ref{taylorexpansion}) by setting $f(x)=\mu_A(x)$ and $L(u)=K(u)$, $L(u)=K(u)\mathds{1}(u\geq 0)$, $L(u)=K(u)u$ and $L(u)=K(u)u\mathds{1}(u\geq 0)$, respectively. Inserting the obtained expansions into (\ref{toinsert}) delivers
\begin{align*}
&\E(K_h(X_i)V_iA)\\
=\;&\underbrace{\begin{pmatrix}\Kpn\\\Kpn\\\Kpe\\\Kpe\end{pmatrix}\mu_{A+}f_X(0)+\begin{pmatrix}\Kmn\\0\\\Kme\\0\end{pmatrix}\mu_{A-}f_X(0)}_{=:S_1}+\underbrace{h\left[ \begin{pmatrix}\Kpe\\\Kpe\\\Kpz\\\Kpz\end{pmatrix}[\mu_Af_X]'_++\begin{pmatrix}\Kme\\0\\\Kmz\\0\end{pmatrix}[\mu_Af_X]'_- \right]}_{=:S_2}\\
    &+\underbrace{\frac{h^2}{2}\left[ \begin{pmatrix}\Kpz\\\Kpz\\\Kpd\\\Kpd\end{pmatrix}[\mu_Af_X]''_++\begin{pmatrix}\Kmz\\0\\\Kmd\\0\end{pmatrix}[\mu_Af_X]''_- \right]+O(h^3).}_{=:S_3}
\end{align*}
We consider all three summands separately. First remember the above properties of the kernel and the fact that $\mu_{A+}=\tau_A+\mu_{A-}$. Then
\begin{equation*}
    \kappa(K)^{-1}S_1=f_X(0)\kappa(K)^{-1} \left[\rule{0cm}{1.3cm}\right.\tau_A\underbrace{\begin{pmatrix}\Kpn\\\Kpn\\\Kpe\\\Kpe\end{pmatrix}}_{=:C_2}+\mu_{A-}\underbrace{\begin{pmatrix}1\\\Kpn\\0\\\Kpe\end{pmatrix}}_{=:C_1} \left]\rule{0cm}{1.3cm}\right.=f_X(0)\begin{pmatrix}\mu_{A-}\\\tau_A\\0\\0\end{pmatrix}
\end{equation*}
whereby the last equality holds since $C_1$ is the first column and $C_2$ the second column of $\kappa(K)$.
Secondly, we obtain for the second summand
\begin{align*}
    \kappa(K)^{-1}S_2&=h\kappa(K)^{-1}\left[\rule{0cm}{1.3cm}\right. \underbrace{\begin{pmatrix}\Kpe\\\Kpe\\\Kpz\\\Kpz\end{pmatrix}}_{=:C_4}[\mu_Af_X]'_++\left(\rule{0cm}{1.3cm}\right. \underbrace{\begin{pmatrix}0\\\Kpe\\\Kz\\\Kpz\end{pmatrix}}_{=:C_3}-\underbrace{\begin{pmatrix}\Kpe\\\Kpe\\\Kpz\\\Kpz\end{pmatrix}}_{=C_4} \left)\rule{0cm}{1.3cm}\right.[\mu_Af_X]'_- \left]\rule{0cm}{1.3cm}\right.\\&=h\begin{pmatrix}0\\0\\ [\mu_Af_X]'_- \\ [\mu_Af_X]'_+-[\mu_Af_X]'_- \end{pmatrix}
\end{align*}
because $C_3$ and $C_4$ are the third and fourth column of $\kappa(K)$ respectively.
Lastly, the third summand gives
\begin{equation*}
    \kappa(K)^{-1}S_3=h^2B(K,A)+O(h^3).
\end{equation*}
By taking everything together, we obtain
\begin{equation*}
    \kappa(K)^{-1}\E(K_h(X_i)V_iA)=\begin{pmatrix}
    f_X(0)\mu_{A-}\\f_X(0)\tau_A\\h[\mu_Af_X]'_-\\h([\mu_Af_X]'_+-[\mu_Af_X]'_-)
    \end{pmatrix}+h^2B(K,A)+O(h^3)
\end{equation*}
which is the assertion of the Lemma.
\end{proof}

\begin{lemma}\label{lemmaA14}
Let $K$ be a kernel with $K^{(4)}<\infty$. Then,
\begin{equation*}
    \kappa(K)^{-1}=\frac{1}{\left(\Kpe\right)^2-\frac{1}{2}\Kpz}\begin{pmatrix}
    -\Kpz&\Kpz&-\Kpe&\Kpe\\
    \Kpz&-2\Kpz&\Kpe&0\\
    -\Kpe&\Kpe&-\frac{1}{2}&\frac{1}{2}\\
    \Kpe&0&\frac{1}{2}&-1
    \end{pmatrix}.
\end{equation*}
Furthermore define
\begin{align*}
    a_1&=\frac{2\left(\Kpz\right)^2-2\Kpe\Kpd}{\Kpz-2\left(\Kpe\right)^2},\quad
    a_2=\frac{\Kpd-2\Kpe\Kpz}{\Kpz-2\left(\Kpe\right)^2},\\
    b_1&=\frac{2\Kpz\Kpd-2\Kpe K^{(4)}_+}{\Kpz-2\left(\Kpe\right)^2},\quad
    b_2=\frac{K^{(4)}_+-2\Kpe\Kpd}{\Kpz-2\left(\Kpe\right)^2},
\end{align*}
then
\begin{equation*}
    \kappa(K)^{-1}\begin{pmatrix}
    0&\Kpe&2\Kpz&\Kpz\\
    \Kpe&\Kpe&\Kpz&\Kpz\\
    2\Kpz&\Kpz&0&\Kpd\\
    \Kpz&\Kpz&\Kpd&\Kpd
    \end{pmatrix}=\begin{pmatrix}
    0&0&a_1&0\\
    0&0&0&a_1\\
    1&0&-a_2&0\\
    0&1&2a_2&a_2
    \end{pmatrix}
\end{equation*}
and
\begin{equation*}
    \kappa(K)^{-1}\begin{pmatrix}
    2\Kpz&\Kpz&0&\Kpd\\
    \Kpz&\Kpz&\Kpd&\Kpd\\
    0&\Kpd&2K^{(4)}_+&K^{(4)}_+\\
    \Kpd&\Kpd&K^{(4)}_+&K^{(4)}_+
    \end{pmatrix}=\begin{pmatrix}
    a_1&0&-b_1&0\\
    0&a_1&2b_1&b_1\\
    -a_2&0&b_2&0\\
    2a_2&a_2&0&b_2
    \end{pmatrix}.
\end{equation*}
\end{lemma}
\begin{proof}
Jensen's inequality implies $\left(\Kpe\right)^2-\frac{1}{2}\Kpz\neq 0$ as shown in the proof of Lemma \ref{matrixinv}. The rest of the statement is a straightforward calculation.
\end{proof}

\begin{lemma}\label{lemmaA15}
Let $f_X$ be twice continuously differentiable in a neighborhood around zero and $K$ be a kernel such that $K^{(4)}<\infty$. For $h\to 0$ and $nh\to\infty$ as $n\to\infty$, we have
\begin{align*}
    &\kappa(K)^{-1}\E(K_h(X_i)V_iV_i^\top)\\
    =\:&f_X(0)I+f_X'(0)h\begin{pmatrix}
    0&0&a_1&0\\
    0&0&0&a_1\\
    1&0&-a_2&0\\
    0&1&2a_2&a_2
    \end{pmatrix}+h^2\frac{f_X''(0)}{2}\begin{pmatrix}
    a_1&0&-b_1&0\\
    0&a_1&2b_1&b_1\\
    -a_2&0&b_2&0\\
    2a_2&a_2&0&b_2
    \end{pmatrix}+o(h^2)
\end{align*}
whereby $a_1,a_2,b_1,b_2$ are defined as in Lemma \ref{lemmaA14}.
\end{lemma}
\begin{proof}
To calculate $\E(K_h(X_i)V_iV_i^\top)$ we again use Taylor expansion, as stated in (\ref{taylorexpansion}), component-wise in combination with the kernel properties in (\ref{K_properties}). Overall, this gives
\begin{align*}
    \kappa(K)^{-1}\E(K_h(X_i)V_iV_i^\top)=\:&f_X(0)I+f_X'(0)h\kappa(K)^{-1}\begin{pmatrix}
    0&\Kpe&2\Kpz&\Kpz\\
    \Kpe&\Kpe&\Kpz&\Kpz\\
    2\Kpz&\Kpz&0&\Kpd\\
    \Kpz&\Kpz&\Kpd&\Kpd
    \end{pmatrix}\\&+\frac{h^2}{2}f_X''(0)\kappa(K)^{-1}\begin{pmatrix}
    2\Kpz&\Kpz&0&\Kpd\\
    \Kpz&\Kpz&\Kpd&\Kpd\\
    0&\Kpd&2K^{(4)}_+&K^{(4)}_+\\
    \Kpd&\Kpd&K^{(4)}_+&K^{(4)}_+
    \end{pmatrix}+o(h^2).
\end{align*}
We can finish the proof by using Lemma \ref{lemmaA14} now.
\end{proof}

\begin{lemma}\label{lemmaA11}
Suppose
\begin{equation*}
    \norm{\E(K_h(X_i)Z_iZ_i^\top)^{-1}}_2=O(1)
\end{equation*}
and assume that $f_X$ is continuous in a neighborhood around zero and $\mu_{ZY}$ can be extended continuously to zero from the left and the right. Then
\begin{equation*}
    \norm{\E(K_h(X_i)Z_iZ_i^\top)^{-1}\E(K_h(X_i)Z_iY_i)}_2=O(1).
\end{equation*}
\end{lemma}
\begin{proof}
By using a well-known property of the operator norm, we obtain
\begin{equation*}
    \norm{\E(K_h(X_i)Z_iZ_i^\top)^{-1}\E(K_h(X_i)Z_iY_i)}_2\leq \norm{\E(K_h(X_i)Z_iZ_i^\top)^{-1}}_2\norm{\E(K_h(X_i)Z_iY_i)}_2.
\end{equation*}
The first factor is $O(1)$ by assumption. We can find an upper bound for the second factor by Taylor expansion:
\begin{equation*}
    \E(K_h(X_i)Z_iY_i)=\E(K_h(X_i)\E(Z_iY_i\mid X_i))=\frac{1}{2}f_X(0)\mu_{ZY-}+\frac{1}{2}f_X(0)\mu_{ZY+}+o(1)=O(1).
\end{equation*}
Note that $\mu_{ZY-}$ and $\mu_{ZY+}$ exist by assumption. This shows the assertion of the lemma.
\end{proof}

\begin{lemma}\label{lemmaA13}
    Let $K$ be a kernel with $K^{(4)}<\infty$, $f_X$ be three times continuously differentiable in a neighborhood around zero with $f_X(0)> 0$ and $nh\to\infty, h\to 0$ as $n\to\infty$. Furthermore, suppose that $\mu_Z$ is continuous and one-sided differentiable at zero up to order three whereby the derivatives extend continuously to zero with $\lim_{x\nearrow 0} \mu'_Z(x)=\lim_{x\searrow 0} \mu'_Z(x)=0$ and $\mu_Z(0)=0$. Also, let
    \begin{equation}\label{ZZTassumption}
        \norm{\E(K_h(X_i)Z_iZ_i^\top)^{-1}}_2=O(1).
    \end{equation}
    Then,
    \begin{align*}
        &\left[ \left( I-\E(K_h(X_i)V_iV_i^\top)^{-1}\E(K_h(X_i)V_iZ_i^\top)\E(K_h(X_i)Z_iZ_i^\top)^{-1}\E(K_h(X_i)Z_iV_i^\top) \right)^{-1} \right]_2\\
        =\;&\begin{pmatrix}
            0&1&0&0
        \end{pmatrix} + O(h^4)
    \end{align*}
\end{lemma}
\begin{proof}
Let $k\in\{1,...,p\}$ and define the matrix $\gamma_n\in\mathbb{R}^{p\times 4}$ by its rows via
\begin{align*}
    [\gamma_n]_{k,\cdot}:=\;&\E\left(K_h(X_i)V_iZ_i^{(k)}\right)^\top=\E\left(K_h(X_i)V_i\E\left(Z_i^{(k)}\mid X_i\right)\right)^\top
\end{align*}
Analogously to previous proofs, we can examine each component by conducting a Taylor expansion as in (\ref{taylorexpansion}). In addition, we also use the properties in (\ref{K_properties}) as well as the assumption that $\mu_Z(0)=0$ and $\mu'_{Z+}=\mu'_{Z-}=0$. Then, this results in
\begingroup
\allowdisplaybreaks
\begin{align*}
    [\gamma_n]_{k,\cdot}^\top=\frac{h^2}{2}\left(\begin{pmatrix}
        [\mu_{Z^{(k)}}f_X]''_-&[\mu_{Z^{(k)}}f_X]''_+-[\mu_{Z^{(k)}}f_X]''_-
    \end{pmatrix}\begin{pmatrix}
        \Kz&\Kpz&0&\Kpd\\
        \Kpz&\Kpz&\Kpd&\Kpd
    \end{pmatrix}\right)^\top+O(h^3).
\end{align*}
\endgroup
Thus, we know that
\begin{equation*}
    \norm{\gamma_n}_2=O(h^2)+O(h^3)=O(h^2).
\end{equation*}
This in turn gives us
\begin{equation}\label{Ohvier}
\begin{split}
    &\norm{\E(K_h(X_i)V_iZ_i^\top)\E(K_h(X_i)Z_iZ_i^\top)^{-1}\E(K_h(X_i)Z_iV_i^\top)}_2\\=\;&\norm{\gamma_n^\top\E(K_h(X_i)Z_iZ_i^\top)^{-1}\gamma_n}_2\leq \underbrace{\norm{\gamma_n^\top}_2}_{=O(h^2)}\underbrace{\norm{\E(K_h(X_i)Z_iZ_i^\top)^{-1}}_2}_{=O(1)}\underbrace{\norm{\gamma_n}_2}_{=O(h^2)}=O(h^4)
\end{split}
\end{equation}
whereby we used the submultiplicativity of the spectral norm. Also, we know by Lemma \ref{lemmaA15} that
\begin{equation}\label{Oheins}
    \E(K_h(X_i)V_iV_i^\top)^{-1}=\left[f_X(0)\kappa(K)+O(h)\right]^{-1}=\frac{1}{f_X(0)}\kappa(K)^{-1}+O(h)=O(1).
\end{equation}
Taking the approximations (\ref{Ohvier}) and (\ref{Oheins}) together, we obtain
\begin{align*}
    &I-\underbrace{\E(K_h(X_i)V_iV_i^\top)^{-1}}_{=O(1)}\underbrace{\gamma_n^\top\E(K_h(X_i)Z_iZ_i^\top)^{-1}\gamma_n}_{=O(h^4)}=I+O(h^4)
\end{align*}
Consequently, we finally obtain the assertion of the lemma, i.e.
\begin{equation*}
    \left[ \left( I-\E(K_h(X_i)V_iV_i^\top)^{-1}\gamma_n^\top\E(K_h(X_i)Z_iZ_i^\top)^{-1}\gamma_n \right)^{-1} \right]_2=\begin{pmatrix}
        0&1&0&0
    \end{pmatrix}+O(h^4).
\end{equation*}
\end{proof}

\newpage
\section{Convergence to the Standard Normal Distribution}

\begin{lemma}\label{lemmaConvergence}
Let $K$ be a kernel which is compactly supported on $[-1,1]$ and satisfies $(K^2)^{(0)}<\infty$. Let $f_X$ be continuous in a neighborhood around zero and $h\to 0, nh\to \infty$ as $n\to\infty$. Moreover, let $A_i\in\mathbb{R}, i=1,...,n$ be random variables, which can also depend on $n$ and $h$, such that $\left( (X_i,A_i) \right)_{i=1,...,n}$ is a family of $n$ independent random variables and
\begin{equation}\label{assumpL61}
    \sup_{n\in\mathbb{N}}\sup_{x\in [-h,h]} \E(A_i^2\mid X_i=x)<\infty.
\end{equation}
Then,
\begin{equation*}
    \frac{1}{n}\sum_{i=1}^n K_h(X_i)A_i=\E(K_h(X_i)A_i)+O_P\left(\frac{1}{\sqrt{nh}}\right).
\end{equation*}
\end{lemma}
\begin{proof}
We set
\begin{equation*}
    S_n:=\frac{1}{n}\sum_{i=1}^n K_h(X_i)A_i.
\end{equation*}
By the imposed assumptions, we know that the mean and variance of $S_n$ are finite. Therefore, we can use Chebyshev's inequality in order to obtain
\begingroup
\allowdisplaybreaks
\begin{align*}
    &P(|S_n-\E(S_n)|>M)\\
    \leq\;&\frac{1}{M^2}\E\left( (S_n-\E(S_n))^2 \right)\\
    =\;&\frac{1}{M^2}\frac{1}{n^2}\E\left( \left( \sum_{i=1}^n \left(K_h(X_i)A_i-\E(K_h(X_i)A_i)\right) \right)^2 \right)\\
    =\;&\frac{1}{M^2}\frac{1}{n^2}\E\left( \sum_{i=1}^n (K_h(X_i)A_i-\E(K_h(X_i)A_i))^2\right) + \sum_{i\neq j} \underbrace{\mathrm{Cov}(K_h(X_i)A_i,K_h(X_j)A_j)}_{=0}\\
    =\;&\frac{1}{M^2}\frac{1}{n^2}\sum_{i=1}^n \E\left( (K_h(X_i)A_i-\E(K_h(X_i)A_i))^2 \right)\\
    \leq\;&\frac{1}{M^2}\frac{1}{n}\E\left( (K_h(X_i)A_i)^2 \right)\\
    =\;&\frac{1}{M^2}\frac{1}{n}\E\left( K_h(X_i)^2\E\left( A_i^2 \mid X_i\right)\right)\\
    \leq\;&\sup_{n\in\mathbb{N}}\sup_{x\in [-h,h]}\E\left( A_i^2 \mid X_i=x\right)\frac{1}{M^2}\frac{1}{nh}\int_{-1}^1 K\left(y\right)^2f_X(yh)\,dy=O\left(\frac{1}{nhM^2}\right),
\end{align*}
\endgroup
whereby $\int_{-1}^1 K\left(y\right)^2f_X(yh)\,dy=O(1)$ due to $(K^2)^{(0)}<\infty$, $h\to 0$ and the continuity of $f_X$ in a neighborhood around zero.
Overall, this leads to
\begin{equation*}
    P\left(|S_n-\E(S_n)|>M\right)\leq \frac{C}{nhM^2}
\end{equation*}
for some $C>0$. Given an arbitrary $\epsilon>0$, choose $M$ large enough such that $\frac{C}{M^2}<\epsilon$. Then,
\begin{equation*}
    P(\sqrt{nh}|S_n-\E(S_n)|>M)=P\left(|S_n-\E(S_n)|>\frac{M}{\sqrt{nh}}\right)\leq\frac{C}{M^2}<\epsilon,
\end{equation*}
which means that $S_n-\E(S_n)=O_P\left(\frac{1}{\sqrt{nh}}\right)$.
\end{proof}

\begin{lemma}\label{lemmaA8}
Let $K$ be a kernel which is compactly supported on $[-1,1]$ and satisfies $(K^2)^{(0)}<\infty$. Let $f_X$ be continuous in a neighborhood around zero. Furthermore, let $h\to 0, nh\to \infty$ as $n\to\infty$ and
\begin{equation}\label{rsquared}
    \sup_{n\in\mathbb{N}}\sup_{x\in[-h,h]} \E(r_i(h)^2\mid X_i=x)<\infty.
\end{equation}
Then,
\begin{equation*}
    \frac{1}{n}\Vbf^\top \Kh \rbf(h)=O_P\left(\frac{1}{\sqrt{nh}}\right).
\end{equation*}
\end{lemma}
\begin{proof}
We want to prove this row-wise. Let $a\in\{1,...,4\}$, then
\begin{align*}
    \E\left( \frac{1}{h}K\left(\frac{X_i}{h}\right)^2 V_{i,a}^2 r_i(h)^2 \right)&=\E\left(\E\left( \frac{1}{h}K\left(\frac{X_i}{h}\right)^2 V_{i,a}^2 r_i(h)^2 \;\,\vrule\; X_i\right)\right)\\
    &=\E\left( \frac{1}{h}K\left(\frac{X_i}{h}\right)^2 V_{i,a}^2 \E(r_i(h)^2\mid X_i) \right)=O(1)
\end{align*}
due to assumption (\ref{rsquared}) and our assumptions imposed on $f_X$ and $K$. Using Markov's inequality, we obtain for every $\epsilon>0$ that
\begin{align*}
    P\left( \left| \frac{1}{n}\Vbf^\top_{\cdot a}\Kh\rbf(h) \right|>M \right)&\leq\frac{1}{M^2}\E\left( \left( \frac{1}{n}\sum_{i=1}^n \frac{1}{h}K\left(\frac{X_i}{h}\right) V_{i,a} r_i(h) \right)^2 \right)\\
    &=\frac{1}{nhM^2}\E\left( \frac{1}{h}K\left(\frac{X_i}{h}\right)^2 V_{i,a}^2 r_i(h)^2 \right)=O\left( \frac{1}{nhM^2} \right),
\end{align*}
leading to
\begin{align*}
    P\left( \left| \frac{1}{n}\Vbf^\top_{\cdot a}\Kh\rbf(h) \right|> M \right)\leq \frac{C}{nhM^2}
\end{align*}
for some constant $C>0$. Now we can conclude with the same argument as in the proof of Lemma \ref{lemmaConvergence}.
\end{proof}

\begin{lemma}\label{propositionA7}
Let all assumptions of Lemma \ref{lemmaA13} hold. Also, let $K$ be a kernel which is compactly supported on $[-1,1]$ and satisfies $K^{(4)}, (K^2)^{(0)}<\infty$. Additionally, define
\begin{equation*}
    M:=I-\Khfrac \Zbf(\Ztbf \Kh \Zbf)^{-1}\Ztbf \Khfrac
\end{equation*}
and suppose that for all $k, l\in\{1,...,p\}$
\begin{equation}\label{cond1}
\begin{split}
    &\sup_{n\in\mathbb{N}}\sup_{x\in[-h,h]}\, \E\left(\left(Z_i^{(k)}r_i(h)\right)^2\;\,\vrule\; X_i=x\right) < \infty,\\
    &\sup_{n\in\mathbb{N}}\sup_{x\in[-h,h]}\, \E\left(\left(Z_i^{(k)}Z_i^{(l)}\right)^2\;\,\vrule\; X_i=x\right)<\infty.
\end{split}
\end{equation}
Then
\begin{equation*}
    \frac{1}{n}\Vbf^\top \Khfrac M\Khfrac\Vbf = f_X(0)\kappa(K)+o_P(1)
\end{equation*}
and
\begin{equation*}
    \frac{1}{n}\Vbf^\top \Khfrac M\Khfrac \rbf(h)=\frac{1}{n}\Vbf^\top \Kh \rbf(h)+o_P\left(\frac{1}{\sqrt{nh}}\right).
\end{equation*}
\end{lemma}
\begin{proof}
The idea of the proof is to use Lemma \ref{lemmaConvergence} multiple times. Indeed, we can do this by calculating
\begin{align*}
\frac{1}{n}\Vbf^\top \Khfrac M\Khfrac\Vbf&=\frac{1}{n}\Vbf^\top \Khfrac\left(I-\Khfrac \Zbf(\Ztbf \Kh \Zbf)^{-1}\Ztbf \Khfrac\right)\Khfrac\Vbf\\
&=\frac{1}{n}\Vbf^\top\Kh\Vbf-\frac{1}{n}\Vbf^\top\Kh\Zbf(\Ztbf \Kh \Zbf)^{-1}\Ztbf\Kh\Vbf\\
&=\underbrace{\frac{1}{n}\Vbf^\top\Kh\Vbf}_{(1)}-\underbrace{\left(\frac{1}{n}\Vbf^\top\Kh\Zbf\right)\left(\frac{1}{n}\Ztbf \Kh \Zbf\right)^{-1}\left(\frac{1}{n}\Ztbf\Kh\Vbf\right)}_{(2)}.
\end{align*}
Concerning term $(2)$, we can now estimate each of the three factors individually. For the first factor we obtain
\begin{equation*}
    \frac{1}{n}\Vbf^\top\Kh\Zbf=\frac{1}{n}\sum_{i=1}^n K_h(X_i)V_iZ_i^\top,
\end{equation*}
leading to the fact that we can consider this expression entry-wise and, respectively, set $A_i$ of Lemma \ref{lemmaConvergence} to every entry of $V_iZ_i^\top$. The necessary assumption (\ref{assumpL61}) for that is implied by (\ref{cond1}) as we have for $m\in\{1,...,4\}$, $k\in\{1,...,p\}$ and $u(x):=(1, \mathds{1}(x\geq 0), x/h, x\mathds{1}(x\geq 0)/h)^\top$ that
\begin{align*}
    &\sup_{n\in\mathbb{N}}\sup_{x\in[-h,h]}\, \E\left(\left(V_i^{(m)}Z_i^{(k)}\right)^2\;\,\vrule\; X_i=x\right)\\
    =\;& \sup_{n\in\mathbb{N}}\sup_{x\in[-h,h]}\, u^{(m)}(x)^2\E\left(\left(Z_i^{(k)}\right)^2\;\,\vrule\; X_i=x\right)\\
    \leq\;& \sup_{n\in\mathbb{N}}\sup_{x\in[-h,h]}\, \E\left(\left(Z_i^{(k)}\right)^2\;\,\vrule\; X_i=x\right)<\infty.
\end{align*}
As result, we obtain
\begin{equation}\label{factorOne}
\frac{1}{n}\Vbf^\top\Kh\Zbf=\E(K_h(X_i)V_iZ_i^\top)+O_P\left(\frac{1}{\sqrt{nh}}\right).
\end{equation}
In an analogous way, the third factor results in
\begin{equation}\label{factorThree}
    \frac{1}{n}\Ztbf\Kh\Vbf=\E(K_h(X_i)Z_iV_i^\top)+O_P\left(\frac{1}{\sqrt{nh}}\right).
\end{equation}
For the second factor we have
\begin{equation}\label{factorTwo}
\begin{split}
    \left(\frac{1}{n}\Ztbf \Kh \Zbf\right)^{-1}&=\left(\E(K_h(X_i)Z_iZ_i^\top)+O_P\left(\frac{1}{\sqrt{nh}}\right)\right)^{-1}\\
    &\overset{(*)}{=}\E(K_h(X_i)Z_iZ_i^\top)^{-1}+O_P\left(\frac{1}{\sqrt{nh}}\right),
\end{split}
\end{equation}
whereby $(*)$ can be concluded using assumption (\ref{ZZTassumption}) and $\sqrt{nh}\to\infty$.

Overall, taking the equations (\ref{factorOne}), (\ref{factorThree}) and (\ref{factorTwo}) together, $(2)$ evaluates as follows:
\begin{align*}
&\left(\frac{1}{n}\Vbf^\top\Kh\Zbf\right)\left(\frac{1}{n}\Ztbf \Kh \Zbf\right)^{-1}\left(\frac{1}{n}\Ztbf\Kh\Vbf\right)\\
=&\;\underbrace{\E(K_h(X_i)V_iZ_i^\top)\E(K_h(X_i)Z_iZ_i^\top)^{-1}\E(K_h(X_i)Z_iV_i^\top)}_{(\triangle)}+o_P(1).
\end{align*}
Note that we used the fact that all the means are of order $O(1)$ because $\E(K_h(X_i)Z_iZ_i^\top)^{-1}=O(1)$ by assumption and $\E(K_h(X_i)Z_iV_i^\top)^\top=\E(K_h(X_i)V_iZ_i^\top)=O(1)$ by applying Lemma \ref{lemmaA12} as in (\ref{taylorOfVZ}). As we now use equation (\ref{Ohvier}) taken from Lemma \ref{lemmaA13}, which states that expression $(\triangle)$ is of order $O(h^4)$, we obtain
\begin{equation*}
\left(\frac{1}{n}\Vbf^\top\Kh\Zbf\right)\left(\frac{1}{n}\Ztbf \Kh \Zbf\right)^{-1}\left(\frac{1}{n}\Ztbf\Kh\Vbf\right)=O(h^4)+o_P(1)=o_P(1).
\end{equation*}

Now we work on term $(1)$. Also here, we apply Lemma \ref{lemmaConvergence} component-wise. It is clear that the necessary assumption (\ref{assumpL61}) is satisfied. Afterwards, we apply Lemma \ref{lemmaA15}, leading to
\begin{align*}
\frac{1}{n}\Vbf^\top\Kh\Vbf&=\E(K_h(X_i)V_iV_i^\top)+O_P\left(\frac{1}{\sqrt{nh}}\right)\\
&=f_X(0)\kappa(K)+o_P(1)+O_P\left(\frac{1}{\sqrt{nh}}\right)\\
&=f_X(0)\kappa(K)+o_P(1).
\end{align*}
Overall, the calculations for term $(1)$ and $(2)$ result in
\begin{equation*}
    \frac{1}{n}\Vbf^\top \Khfrac M\Khfrac\Vbf=f_X(0)\kappa(K)+o_P(1).
\end{equation*}
Now we have to show the second assertion of the lemma. We start in the same way by splitting up the expression:
\begin{align*}
\frac{1}{n}\Vbf^\top \Khfrac M\Khfrac \rbf(h)&=\frac{1}{n}\Vbf^\top \Khfrac\left(I-\Khfrac \Zbf(\Ztbf \Kh \Zbf)^{-1}\Ztbf \Khfrac\right)\Khfrac\rbf(h)\\
&=\frac{1}{n}\Vbf^\top\Kh\rbf(h)-\frac{1}{n}\Vbf^\top\Kh\Zbf(\Ztbf \Kh \Zbf)^{-1}\Ztbf\Kh\rbf(h)\\
&=\underbrace{\frac{1}{n}\Vbf^\top\Kh\rbf(h)}_{(1)}-\underbrace{\left(\frac{1}{n}\Vbf^\top\Kh\Zbf\right)\left(\frac{1}{n}\Ztbf \Kh \Zbf\right)^{-1}\left(\frac{1}{n}\Ztbf\Kh\rbf(h)\right)}_{(2)}.
\end{align*}
We do not have to examine expression $(1)$ since it appears in the expression of the assertion. We just have to show that term $(2)$ is of order $o_P\left(\frac{1}{\sqrt{nh}}\right)$.
We know that:
\begin{itemize}
    \item $\frac{1}{n}\Vbf^\top\Kh\Zbf=\E(K_h(X_i)V_iZ_i^\top)+O_P\left(\frac{1}{\sqrt{nh}}\right)=o(1)+O_P\left(\frac{1}{\sqrt{nh}}\right)=o_P(1)$ whereby the first equation follows from Lemma \ref{lemmaConvergence} and the second by $\E(K_h(X_i)V_iZ_i^\top)=o(1)$, which for example can be seen in the calculation of (\ref{taylorOfVZ}) or by applying Lemma \ref{lemmaA12}.
    \item $\left(\frac{1}{n}\Ztbf \Kh \Zbf\right)^{-1}=\E(K_h(X_i)Z_iZ_i^\top)^{-1}+O_P\left(\frac{1}{\sqrt{nh}}\right)=O_P(1)$ whereby we used (\ref{factorTwo}) and assumption (\ref{ZZTassumption}).
    \item $\frac{1}{n}\Ztbf\Kh\rbf(h)=\E(K_h(X_i)Z_ir_i(h))+O_P\left(\frac{1}{\sqrt{nh}}\right)=O_P\left(\frac{1}{\sqrt{nh}}\right)$ by applying Lemma \ref{lemmaConvergence} with $A_i$ set respectively to the components of $Z_ir_i(h)$ and the fact that
    \begin{equation*}
        \E(K_h(X_i)Z_ir_i(h))=0
    \end{equation*}
    due to the definition of the residual with $\theta_0(h)$ and $\gamma_0(h)$ being defined as minimizing arguments.
\end{itemize}
Overall, this shows that $(2)$ is of order $o_P\left(\frac{1}{\sqrt{nh}}\right)$, leading to the second assertion of the lemma being proved.
\end{proof}

\end{appendices}

\newpage
\bibliography{bibliography}

\begin{thebibliography}{16}
\providecommand{\natexlab}[1]{#1}
\providecommand{\url}[1]{{#1}}
\providecommand{\urlprefix}{URL }
\providecommand{\doi}[1]{\url{https://doi.org/#1}}
\providecommand{\eprint}[2][]{\url{#2}}
 \bibcommenthead

\bibitem[{Arai et~al(2021)Arai, Otsu, and Seo}]{Arai2021}
Arai Y, Otsu T, Seo MH (2021) {Regression Discontinuity Design with Potentially Many Covariates}. \doi{10.48550/ARXIV.2109.08351}

\bibitem[{Armstrong and Koles{\'{a}}r(2018)}]{Armstrong_2018}
Armstrong TB, Koles{\'{a}}r M (2018) Optimal inference in a class of regression models. Econometrica 86(2):655--683. \doi{10.3982/ecta14434}

\bibitem[{Calonico et~al(2014)Calonico, Cattaneo, and Titiunik}]{KinkRDD}
Calonico S, Cattaneo MD, Titiunik R (2014) {Robust Nonparametric Confidence Intervals for Regression-Discontinuity Designs}. Econometrica \doi{10.3982/ECTA11757}

\bibitem[{Calonico et~al(2019)Calonico, Cattaneo, Farrell, and Titiunik}]{RDDWithCov}
Calonico S, Cattaneo MD, Farrell MH, et~al (2019) {Regression Discontinuity Designs Using Covariates}. The Review of Economics and Statistics \doi{10.1162/rest_a_00760}

\bibitem[{Cattaneo et~al(2017)Cattaneo, Titiunik, and Vazquez-Bare}]{ComparingInference}
Cattaneo MD, Titiunik R, Vazquez-Bare G (2017) {Comparing Inference Approaches for RD Designs: a Reexamination of the Effect of Head Start on Child Mortality}. Journal of Policy Analysis and Management 36:643--681. \doi{10.1002/pam.21985}

\bibitem[{Cattaneo et~al(2019)Cattaneo, Idrobo, and Titiunik}]{RDFoundations}
Cattaneo MD, Idrobo N, Titiunik R (2019) {A Practical Introduction to Regression Discontinuity Designs: Foundations}. Cambridge Elements: Quantitative and Computational Methods for Social Science \doi{10.1017/9781108684606}

\bibitem[{Cattaneo et~al(2023)Cattaneo, Idrobo, and Titiunik}]{RDExtensions}
Cattaneo MD, Idrobo N, Titiunik R (2023) {A Practical Introduction to Regression Discontinuity Designs: Extensions}. Cambridge Elements: Quantitative and Computational Methods for Social Science {\href{https://arxiv.org/abs/2301.08958}{{2301.08958}}}

\bibitem[{Frölich and Huber(2019)}]{Froe2019}
Frölich M, Huber M (2019) {Including Covariates in the Regression Discontinuity Design}. Journal of Business \& Economic Statistics 37(4):736--748. \doi{10.1080/07350015.2017.1421544}

\bibitem[{Gerard et~al(2020)Gerard, Rokkanen, and Rothe}]{GRR20}
Gerard F, Rokkanen M, Rothe C (2020) Bounds on treatment effects in regression discontinuity designs with a manipulated running variable. Quantitative Economics 11(3):839--870. \doi{10.3982/QE1079}

\bibitem[{Hahn et~al(2001)Hahn, Todd, and der Klaauw}]{Hahn2001}
Hahn J, Todd P, der Klaauw WV (2001) {Identification and Estimation of Treatment Effects with a Regression-Discontinuity Design}. Econometrica 69(1):201--209. \urlprefix\url{http://www.jstor.org/stable/2692190}

\bibitem[{Imbens and Kalyanaraman(2011)}]{RDMSEBandwidth}
Imbens G, Kalyanaraman K (2011) {Optimal Bandwidth Choice for the Regression Discontinuity Estimator}. The Review of Economic Studies 79(3):933--959. \doi{10.1093/restud/rdr043}

\bibitem[{Imbens and Lemieux(2008)}]{IMBENS2008615}
Imbens GW, Lemieux T (2008) Regression discontinuity designs: A guide to practice. Journal of Econometrics 142(2):615--635. \doi{10.1016/j.jeconom.2007.05.001}

\bibitem[{Jin et~al(2023)Jin, Ren, and Candès}]{JRC23}
Jin Y, Ren Z, Candès EJ (2023) Sensitivity analysis of individual treatment effects: A robust conformal inference approach. Proceedings of the National Academy of Sciences 120(6):e2214889120. \doi{10.1073/pnas.2214889120}

\bibitem[{Kreiß and Rothe(2021)}]{RDWithPotManyCov}
Kreiß A, Rothe C (2021) {Inference in Regression Discontinuity Designs with High-Dimensional Covariates}. \doi{10.48550/ARXIV.2110.13725}

\bibitem[{Lee and Lemieux(2010)}]{LeeDavid2010}
Lee DS, Lemieux T (2010) {Regression Discontinuity Designs in Economics}. Journal of Economic Literature 48(2):281--355. \doi{10.1257/jel.48.2.281}

\bibitem[{Noack et~al(2021)Noack, Olma, and Rothe}]{Noack2021}
Noack C, Olma T, Rothe C (2021) {Flexible Covariate Adjustments in Regression Discontinuity Designs}. \doi{10.48550/ARXIV.2107.07942}

\end{thebibliography}

\end{document}